\documentclass[11pt,a4paper,lualatex]{amsart}
\usepackage[marginparwidth=0pt,margin=20truemm]{geometry} 

\usepackage{mypackage} 
\usepackage{mycommand} 
\usetikzlibrary{knots}
\usetikzlibrary{positioning}

\DeclareMathOperator{\WRT}{WRT}

\usepackage[pdfencoding=auto,hypertexnames=false]{hyperref} 
\hypersetup{colorlinks=false}

\numberwithin{equation}{section} 
\usepackage{mytheoremeng} 

\begin{document}

\title[WRT invariants and indefinite false theta functions]{Witten--Reshetikhin--Turaev invariants and indefinite false theta functions for plumbing indefinite H-graphs}
\author[Y. Murakami]{Yuya Murakami}
\address{Mathematical Inst. Tohoku Univ., 6-3, Aoba, Aramaki, Aoba-Ku, Sendai 980-8578, JAPAN}
\email{murakami.yuya.896@m.kyushu-u.ac.jp}

\date{\today}


\begin{abstract}
	Gukov--Pei--Putrov--Vafa conjectured the existence of $ q $-series whose radial limits are Witten--Reshetikhin--Turaev invariants and called them homological blocks.
	For weakly negative definite plumbed 3-manifolds, Gukov--Pei--Putrov--Vafa and Gukov-Manolescu constructed homological blocks.
	In this paper, we construct indefinite false theta functions which are candidates of homological blocks for some plumbed $ 3 $-manifolds which are not weakly negative definite.
	Moreover we prove that, for the Poincar\'{e} homology sphere, our indefinite false theta function coincides with the original homological block.
\end{abstract}


\maketitle
\tableofcontents


\section{Introduction} \label{sec:intro}


The Witten--Reshetikhin--Turaev (WRT) invariants are important quantum invariants of $ 3 $-manifolds constructed by Witten~\cite{Witten} and Reshetikhin--Turaev~\cite[Theorem 3.3.2]{Reshetikhin-Turaev} from physical and mathematical viewpoints respectively.
It is important to study the asymptotic expansions of the WRT invariants.
The Witten's asymptotic expansion conjecture claims that such expansions can be written by the Chern--Simons invariants and the Reidemeister torsions.
The first progress on this conjecture was made by Lawrence--Zagier~\cite{LZ}.
They proved the conjecture for the Poincar\'{e} homology sphere.
The key idea of their proof is constructing a false theta function whose radial limits coincide with the WRT invariants of the Poincar\'{e} homology sphere.
Then, we obtain the asymptotic expansion of WRT invariants by considering the modular transformation of such a false theta function.

Here we remark that such an argument leads us to the concept of quantum modular forms introduced by Zagier~\cite{Zagier_quantum}.
Although quantum modular forms are purely number theoretical objects, many examples arise from $ 3 $-dimensional topology.
The WRT invariants of the Poincar\'{e} homology sphere is a typical example of quantum modular forms.
Other examples are colored Jones polynomials.
For torus knots, their quantum modularity was studied by Hikami~\cite{Hikami_AGid}, Hikami--Kirillov~\cite{Hikami-Kirillov_torus,Hikami-Kirillov_L-function}, and Hikami--Lovejoy~\cite{Hikami-Lovejoy_QMF}.
Galoufalidis--Zagier~\cite{Garoufalidis-Zagier} studied their quantum modularity for hyperbolic knots.

We can conclude that the proof by Lawrence--Zagier~\cite{LZ} is based on a proof of the quantum modularity of the WRT invariants.
Thus, we can expect that we can prove the Witten's asymptotic expansion conjecture by proving the quantum modularity of the WRT invariants.
To prove it, we need to construct $ q $-series associated with $ 3 $-manifolds whose radial limits coincide with the WRT invariants.
Hikami~\cite{H_Bries,H_Seifert} constructed such $ q $-series and proved the Witten's conjecture for Brieskorn homology spheres and Seifert homology spheres, respectively. 
Fuji--Iwaki--Murakami--
Terashima~\cite{FIMT} also constructed such $ q $-series for Seifert homology spheres.
The most general construction of candidates for such $ q $-series is carried out by Gukov--Pei--Putrov--Vafa~\cite{GPPV} and Gukov--Manolescu~\cite{GM}.
They constructed $ q $-series invariants called homological blocks or GPPV invariants for plumbed manifolds with weakly negative definite linking matrices.
Andersen--Misteg{\aa}rd~\cite{Andersen-Mistegard} proved that homological blocks for Seifert homology spheres coincide with $ q $-series constructed by Fuji--Iwaki--Murakami--Terashima~\cite{FIMT}.
Many other authors have studied the Witten's asymptotic expansion conjecture~\cite{Andersen-Hansen,Andersen-Himpel,Andersen-Mistegard,Andersen,Andersen-Petersen,Beasley-Witten,Charles,Chun,Chung_Seifert,Chung_rational,Chung_resurgent,Chung_SU(N),Charles-Marche_I,Charles-Marche_II,FIMT,Freed-Gompf,GMP,H_Lattice,H_Lattice2,Hansen-Takata,Jeffrey,Rozansky1,Rozansky2,Wu}.

Many authors studied the modular transformation of homological blocks.
Matsusaka--Terashima~\cite{Matsusaka-Terashima} proved it for Seifert homology spheres.
Bringmann--Mahlburg--Milas~\cite{BMM_high_depth} proved it for non-Seifert homology spheres whose surgery diagrams are the H-graphs with negative definite linking matrices.
Murakami~\cite[Proposition 4.2]{M_plumbed} proved that homological blocks are false theta functions for plumbed homology spheres.
Bringmann--Nazaroglu~\cite{Bringmann-Nazaroglu} and Bringmann--Kaszian--Milas--Nazaroglu~\cite{BKMN_False_modular} clarified and proved the modular transformation formulas of false theta functions.

Gukov--Pei--Putrov--Vafa~\cite[Equation (A.28)]{GPPV} conjectured that radial limits of homological blocks are WRT invariants.
For Seifert homology spheres, this conjecture is confirmed by 
Andersen--Misteg{\aa}rd~\cite{Andersen-Mistegard} and 
Fuji--Iwaki--Murakami--Terashima~\cite{FIMT} with a result in
Andersen--Misteg{\aa}rd~\cite{Andersen-Mistegard} for Seifert homology spheres independently.
Mori--Murakami~\cite{MM} proved this conjecture for non-Seifert homology spheres whose surgery diagrams are the H-graphs with negative definite linking matrices.
Murakami~\cite{M_plumbed,M_GPPV} proved it for all negative definite plumbed manifolds.


All of these studies dealt with weakly definite plumbed manifolds.
In this paper, we deal with the simplest manifolds which are not such manifolds, that is, plumbed homology spheres whose surgery diagrams are the H-graphs with indefinite linking matrices.
For such manifolds, we construct indefinite false theta functions and prove that their radial limits coincide with the WRT invariants.
Here, indefinite false theta functions are ``false''  objects of indefinite theta functions in the same way that false theta functions are ``false''  objects of theta functions.
Indefinite theta functions are infinite series associated with indefinite quadratic forms. In contrast, theta functions are infinite series associated with positive definite quadratic forms.
Zwegers~\cite{Zwegers_thesis} introduced one of the most important class of indefinite theta functions, and he proved their modular transformation laws.
Thus, Zwegers' indefinite theta functions are the first candidates of our $ q $-series whose radial limits are WRT invariants.
However, it turns out that their radial limits vanish in this article.
That is the reason why we consider indefinite false theta functions.

Some authors studied indefinite false theta functions.
Andrews--Dyson--Hickerson~\cite{Andrews-Dyson-Hickerson} expressed Ramanujan's $ q $-hypergeometric functions $ \sigma(q) $ as an indefinite false theta function and proved that it is related to the real quadratic field $ \Q(\sqrt{6}) $.
Lovejoy--Osburn~\cite{Lovejoy-Osburn} considered a dozen $ q $-hypergeometric functions and proved they are related to real quadratic fields.
Recently, Bringmann--Nazaroglu~\cite{Bringmann-Nazaroglu_real_quad_double_sum} obtained the modular transformation law of Lovejoy--Osburn's $ q $-hypergeometric functions.

Let us explain our setting for our main result.
Let $ \Gamma $ be the H-graph with vertices $ 1, \dots, 6 $ and weights $ w_v \in \Z $ for each vertex $ v $ shown in \cref{fig:H-graph}.
Here we remark that the surgery diagram $ \calL(\Gamma) $ corresponding to $\Gamma$ is shown in \cref{fig:surgery_diagram} and its linking matrix is
\[
W = \pmat{w_1 & 1 & 1 & 1 & 0 & 0 \\
	1 & w_2 & 0 & 0 & 1 & 1 \\
	1 & 0 & w_3 & 0 & 0 & 0 \\
	1 & 0 & 0 & w_4 & 0 & 0 \\
	0 & 1 & 0 & 0 & w_5 & 0 \\
	0 & 1 & 0 & 0 & 0 & w_6 }.
\]
\begin{figure}[htb]
	\begin{minipage}[b]{0.45\linewidth}
		\centering
		\begin{tikzpicture}
			\node[shape=circle,fill=black, scale = 0.4] (1) at (0,0) { };
			\node[shape=circle,fill=black, scale = 0.4] (2) at (1.5,0) { };
			\node[shape=circle,fill=black, scale = 0.4] (3) at (-1,-1) { };
			\node[shape=circle,fill=black, scale = 0.4] (4) at (-1,1) { };
			\node[shape=circle,fill=black, scale = 0.4] (5) at (2.5,1) { };
			\node[shape=circle,fill=black, scale = 0.4] (6) at (2.5,-1) { };
			
			\node[draw=none] (B1) at (0,0.4) {$ w_1 $};
			\node[draw=none] (B2) at (1.5, 0.4) {$ w_2 $};
			\node[draw=none] (B3) at (-0.6,1) {$ w_3 $};
			\node[draw=none] (B4) at (-0.6,-1) {$ w_4 $};
			\node[draw=none] (B5) at (2.1,1) {$ w_5 $};		
			\node[draw=none] (B6) at (2.1,-1) {$ w_6 $};	
			
			\path [-](1) edge node[left] {} (2);
			\path [-](1) edge node[left] {} (3);
			\path [-](1) edge node[left] {} (4);
			\path [-](2) edge node[left] {} (5);
			\path [-](2) edge node[left] {} (6);
		\end{tikzpicture}
		\caption{The H-graph $ \Gamma $} \label{fig:H-graph}		
	\end{minipage}
	\begin{minipage}[b]{0.45\linewidth}
		\centering
		\begin{tikzpicture}[scale=0.5]
			\begin{knot}[
				clip width=5, 
				flip crossing=2, 
				flip crossing=4, 
				flip crossing=6, 
				flip crossing=7, 
				flip crossing=9
				]
				\strand[thick] (0, 0) circle [x radius=3cm, y radius=1.5cm];
				\strand[thick] (0, 0) +(4cm, 0pt) circle [x radius=3cm, y radius=1.5cm];
				\strand[thick] (0, 0) +(-2.5cm, 2cm) circle [radius=1.5cm];
				\strand[thick] (0, 0) +(-2.5cm, -2cm) circle [radius=1.5cm];
				\strand[thick] (0, 0) +(6.5cm, 2cm) circle [radius=1.5cm];
				\strand[thick] (0, 0) +(6.5cm, -2cm) circle [radius=1.5cm];
				
				\node (1) at (0, 2) {$w_{1}$};
				\node (2) at (4, 2) {$w_{2}$};
				\node (3) at (-2.5, 4) {$w_{3}$};
				\node (4) at (-2.5, -4) {$w_{4}$};
				\node (5) at (6.5, 4) {$w_{5}$};
				\node (6) at (6.5, -4) {$w_{6}$};
			\end{knot}
		\end{tikzpicture}
		\caption{The surgery diagram $ \calL(\Gamma) $ corresponding to $\Gamma$} \label{fig:surgery_diagram}
	\end{minipage}
\end{figure}
Let $ M(\Gamma) $ be the plumed 3-manifold obtained from $ S^3 $ through the surgery along the diagram $ \calL(\Gamma) $.
If $ w_i \in \{ \pm 1 \} $ for some $ 3 \le i \le 6 $, then $ M(\Gamma) $ is a Seifert manifold by Neumann's theorem (\cite[Proposition 2.2]{Neumann_Lecture}, \cite[Theorem 3.1]{Neumann_work}).
Let $ S \in \Sym_2(\Q) $ be the upper left $ 2 \times 2 $-submatrix of $ -W^{-1} $.

In this article, we assume that $ W $ has determinant is $ \pm 1 $ and $ S $ is indefinite.
In this case, $ M(\Gamma) $ is an integral homology sphere since $ \det W = \pm 1 $ and $ H_1(M(\Gamma), \Z) \cong \Z^{6} / W(\Z^{6}) $ by use of the Mayer--Vietoris sequence.
We give examples of such graphs in \cref{sec:examples} later. 

For a positive integer $ k $, let $ \WRT_k(M(\Gamma)) $ be the WRT invariant of $ M(\Gamma) $ normalised as $ \WRT_k(S^3) = 1 $. 
We also denote $ \zeta_k \coloneqq e^{2\pi\iu/k} $.

Under the above settings, our main result is as follows.

\begin{thm} \label{thm:main_indefinite}
	Let $ \Gamma $ be an H-graph such that $ W $ has determinant is $ \pm 1 $, $ S $ is indefinite, and $ w_3, w_4, w_5, w_6 \notin \{ \pm 1 \} $.
	Then, we can construct explicitly an indefinite false theta function $ \widehat{Z}_\Gamma \left( q \right) \in q^c \bbC[[q]] $ with some $ c \in \Q $ such that for any positive integer $ k $ it holds
	\[
	\WRT_k(M(\Gamma))
	= \frac{1}{2(\zeta_{2k} - \zeta_{2k}^{-1})}
	\lim_{q \to \zeta_k^{\varepsilon_\Gamma}} \widehat{Z}_\Gamma \left( q \right),
	\]
	where $ \varepsilon_\Gamma \in \{ \pm 1 \}$ is an explicit number determined by $ \Gamma $.
\end{thm}

We will give pricise definitions of $ \varepsilon_\Gamma $ and $ \widehat{Z}_{\Gamma} (q) $ in \cref{sec:fund_data,sec:infinite_series} respectively.

Our method also can apply to homology spheres which have Y-shaped surgery diagrams since they are obtained also by H-graphs.
For example, in \cref{sec:examples} we will calculate the indefinite false theta function $ \widehat{Z}_\Gamma \left( q \right) $ for the Poincar\'{e} homology sphere and prove that it coincides with the homological block of the Poincar\'{e} homology sphere defined by Gukov--Pei--Putrov--Vafa~\cite{GPPV}.

This paper will be organised as follows. 
In \cref{sec:Gauss_sum}, we give properties of Gauss sums which we need to calculate WRT invariants.
In \cref{sec:fund_data}, we prepare some notations for plumbing H-graphs $ \Gamma $, which we use throughout this paper.
In \cref{sec:WRT}, we write WRT invariants as weighted Gauss sums, which are suitable to represent as radial limits of infinite series.
In \cref{sec:asymptotic_formula}, we prove an asymptotic formula which we can apply for indefinite theta functions and indefinite false theta functions.
In \cref{sec:Zwegers_theta}, we prove the vanishings of radial limits of Zwegers' indefinite theta functions.
In \cref{sec:infinite_series}, we introduce indefinite false theta functions $ \widehat{Z}_{\Gamma} (q) $ and prove our main theorem. 
Finally, in \cref{sec:examples}, we give examples of our H-graphs and an indefinite false theta function whose radial limits are WRT invariants of the Poincar\'{e} homology sphere.


\subsection*{Notations} \label{subsec:notation_others}


Throughout this article, we denote $ q $ a complex variable with $ \abs{q} < 1 $.
We fix a positive number $ k \in \Z_{>0} $ and let $ \zeta_k \coloneqq e^{2\pi\iu/k} $.
Let $ [k] \coloneqq \{ 0, \dots, k-1 \} $ and
\begin{align}
	X + [k] 
	&\coloneqq
	\{ \alpha + m \mid \alpha \in X, m \in [k] \},
	\\
	Y + [k]^2 
	&\coloneqq
	\{ \gamma + (m, n) \mid \gamma \in Y, m, n \in [k] \}
\end{align}
for $ X \subset \R $ and $ Y \subset \R^2 $.
For a complex number $ z $, we denote $ \bm{e}(z) \coloneqq e^{2\pi\iu z} $.
For a real number $ x $, let
\begin{align}
	\sgn(x) \coloneqq 
	\begin{cases}
		1 & x \ge 0, \\
		-1 & x < 0,
	\end{cases}
	\qquad
	\sgn_0(x) \coloneqq 
	\begin{cases}
		1 & x > 0, \\
		0 & x = 0, \\
		-1 & x < 0.
	\end{cases}
\end{align}


\section*{Acknowledgement} \label{sec:acknowledgement}


The author would like to show the greatest appreciation to Takuya Yamauchi for giving much advice. 
Toshiki Matsusaka gave me much advice and taught me many related studies, for example, studies of indefinite false theta functions.
The author would like to thank Yuji Terashima, Akihito Mori, Kazuhiro Hikami, and Robert Osburn for giving many comments. 
The author is supported by JSPS KAKENHI Grant Number JP 20J20308.


\section{Properties of Gauss sums} \label{sec:Gauss_sum}


In this section, we prepare properties of Gauss sums which we need to calculate WRT invariants.


\subsection{Reciprocity of Gauss sums} \label{subsec:reciprocity}


The following property is called ``reciprocity of Gauss sums.''
We use this property to calculate WRT invariants.

\begin{prop}[Reciprocity of Gauss sums, {\cite[Theorem 1]{DT}}] \label{prop:reciprocity}
	Let $ L $ be a lattice of finite rank $ n $ equipped with a non-degenerated symmetric $ \Z $-valued bilinear form $ \sprod{\cdot, \cdot} $.
	We write
	\[
	L' \coloneqq \{ y \in L \otimes \R \mid \sprod{x, y} \in \Z \text{ for all } x \in L \} 
	\]
	for the dual lattice.
	Let $ 0 < k \in \abs{L'/L} \Z, u \in \frac{1}{k} L $, 
	and $ h \colon L \otimes \R \to L \otimes \R $ be a self-adjoint automorphism such that $ h(L') \subset L' $ and $ \frac{k}{2} \sprod{y, h(y)} \in \Z $ for all $ y \in L' $.
	Let $ \sigma $ be the signature of the quadratic form $ \sprod{x, h(y)} $.
	Then it holds
	\begin{align}
		&\sum_{x \in L/kL} \bm{e} \left( \frac{1}{2k} \sprod{x, h(x)} + \sprod{x, u} \right)
		= \,
		&\frac{\bm{e}(\sigma/8) k^{n/2}}{\sqrt{\abs{L'/L} \abs{\det h}}}
		\sum_{y \in L'/h(L')} \bm{e} \left( -\frac{k}{2} \sprod{y + u, h^{-1}(y + u)} \right).
	\end{align}
\end{prop}


\subsection{Vanishing results of weighted Gauss sums} \label{subsec:Gauss_sum_vanish}


In this subsection, we prove the following vanishing results of weighted Gauss sums.
This proposition is a generalization of \cite[Proposition 4.2 (ii)]{MM}, which deals with only positive definite quadratic forms. 

\begin{prop} \label{prop:Gauss_sum_vanish}
	Let  $ k, M, N, a, b $, and $ c $ be non-zero integers 
	and put
	\[
	S = \pmat{Ma & MNb \\ MNb & Nc}, \quad
	Q(m, n) = Mam^2 + 2MNbmn + Ncn^2.
	\]
	Let $ \calT \subset \frac{1}{2M}\Z $ be a finite set such that for $ \alpha \in \calT $, it holds $ \gcd(2M \alpha, M) = 1 $ and
	$ M\alpha, M\alpha^2 \bmod \Z $ are independent of $ \alpha $.
	Let $ \chi \colon \calT \to \bbC $ be a map such that
	\[
	\sum_{\alpha \in \calT} \chi(\alpha) = 0.
	\]
	Let $ \beta \in \frac{1}{2N} \Z $.
	Then, it holds
	\[
	\sum_{\alpha \in \calT} \chi(\alpha)
	\sum_{m \in \Z/k\Z}
	\bm{e}\left( \frac{1}{k} Q(m+\alpha, \beta) \right) 
	= 0.
	\]
\end{prop}

To prove this, we need the following lemma.

\begin{lem}[{\cite[Lemma 6.5]{M_plumbed}}] \label{lem:Gauss_sum_indep}
	Let $ k, M \in \Z \smallsetminus \{ 0 \} $ and $ a, b \in \Z $ be integers such that $ \gcd(a, M) = 1 $.
	For $ \alpha \in \frac{1}{2M}\Z $ such that $ \gcd(2M \alpha, M) = 1 $, the complex number 
	\[
	\sum_{\mu \in \Z/k\Z + \alpha} \bm{e} \left( \frac{M}{k} \left( a \mu^2 + b \mu \right) \right)
	\]
	depends only on $ M\alpha^2, M\alpha \bmod \Z $. 
\end{lem}

\begin{proof}[Proof of $ \cref{prop:Gauss_sum_vanish} $]
	To begin with, we remark that the sum for $ m $ in the statement is well-defined since
	\[
	Q(m + k + \alpha, \beta) - Q(m + \alpha, \beta)
	= 2M(m+\alpha) ak + Ma k^2 + 2bMN \beta k
	\equiv 0 \bmod k.
	\]
	The left hand side in the statement is equal to
	\[
	\bm{e}\left( \frac{Nc}{k} \beta^2 \right)
	\sum_{\alpha \in \calT} \chi(\alpha)
	\sum_{m \in \Z/k\Z}
	\bm{e}\left( \frac{M}{k} (a (m + \alpha)^2 + 2bN\beta (m + \alpha) ) \right).
	\]
	The sum for $ m $ is independent of $ \alpha $ by the assumption and \cref{lem:Gauss_sum_indep}.
	Thus, we obtain the statement.
	%
\end{proof}


\section{Notations for graphs} \label{sec:fund_data}


In this section, we prepare notations which we use throughout this article.

To begin with, we prepare notations for H-graphs.
We already introduced some of them in \cref{sec:intro}.

Let $ \Gamma $ be the H-graph with vertices $ 1, \dots, 6 $ and weights $ w_v \in \Z $ for each vertex $ v $ shown in \cref{fig:H-graph}.
Let $ M(\Gamma) $ be the plumed 3-manifold obtained from $ S^3 $ through the surgery along $ \Gamma $.

\begin{rem} \label{rem:Neumann}
	Neumann's theorem (\cite[Proposition 2.2]{Neumann_Lecture}, \cite[Theorem 3.1]{Neumann_work}) states that two plumbing graphs give homeomorphic manifolds if and only if they are related by a sequence of Neumann shown in \cref{fig:Neumann}.
	It is known that a plumbing graph gives a Seifert manifold if it has at most one vertex whose degree is greater than $ 2 $.
	Thus, if $ w_i \in \{ \pm 1 \} $ for some $ 3 \le i \le 6 $, then $ M(\Gamma) $ is a Seifert manifold.
\end{rem}

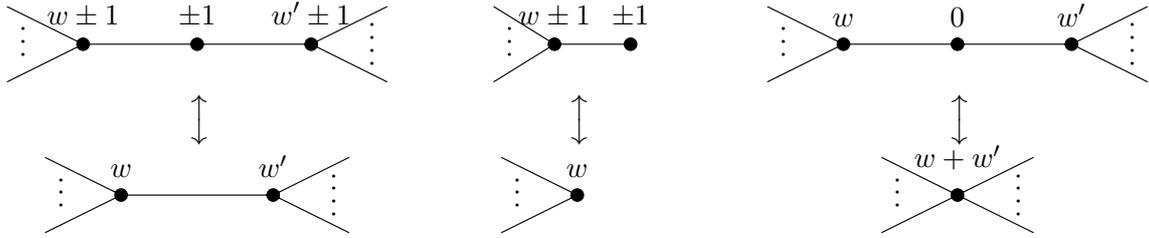
\begin{figure}[htp]
	\centering
	\begin{tikzpicture}
		\draw[fill]
		(-1.5,0) node[above=0.1cm]{$w \pm 1$} circle(0.5ex)--
		(0,0) node[above=0.1cm]{$\pm 1$} circle(0.5ex)--
		(1.5,0) node[above=0.1cm]{$w' \pm 1$} circle(0.5ex)
		(-2.5,0.5) node[above]{}--(-1.5,0) node[above]{}
		(-2.3,0) node[rotate=270]{$\cdots$}
		(-2.5,-0.5) node[above]{}--(-1.5,-0) node[above]{}
		(1.5,0) node[above]{}--(2.5,0.5) node[above]{}
		(2.3,0) node[rotate=270]{$\ldotp\ldotp\ldotp\ldotp$}
		(1.5,0) node[above]{}--(2.5,-0.5) node[above]{}
		(0,-1) node[rotate=270]{$\longleftrightarrow$}
		(-1,-2) node[above=0.1cm]{$ w $} circle(0.5ex)--
		(1,-2) node[above=0.1cm]{$ w' $} circle(0.5ex)
		(-2,-1.5) node[above]{}--(-1,-2) node[above]{}
		(-1.8,-2) node[rotate=270]{$\cdots$}
		(-2,-2.5) node[above]{}--(-1,-2) node[above]{}
		(1,-2) node[above]{}--(2,-1.5) node[above]{}
		(1.8,-2) node[rotate=270]{$\ldotp\ldotp\ldotp\ldotp$}
		(1,-2) node[above]{}--(2,-2.5) node[above]{};
		\draw[fill]
		(4.7,0) node[above=0.1cm]{$w \pm 1$} circle(0.5ex)--
		(5.7,0) node[above=0.1cm]{$\pm 1$} circle(0.5ex)
		(3.9,0.5) node[above]{}--(4.7,0) node[above]{}
		(4.1,0) node[rotate=270]{$\cdots$}
		(3.9,-0.5) node[above]{}--(4.7,0) node[above]{}
		(5,-1) node[rotate=270]{$\longleftrightarrow$}
		(5,-2) node[above=0.1cm]{$w$} circle(0.5ex)
		(4,-1.5) node[above]{}--(5,-2) node[above]{}
		(4.2,-2) node[rotate=270]{$\cdots$}
		(4,-2.5) node[above]{}--(5,-2) node[above]{};
		\draw[fill]
		(8.5,0) node[above=0.1cm]{$w$} circle(0.5ex)--
		(10,0) node[above=0.1cm]{$0$} circle(0.5ex)--
		(11.5,0) node[above=0.1cm]{$w'$} circle(0.5ex)
		(7.5,0.5) node[above]{}--(8.5,0) node[above]{}
		(7.7,0) node[rotate=270]{$\cdots$}
		(7.5,-0.5) node[above]{}--(8.5,0) node[above]{}
		(11.5,0) node[above]{}--(12.5,0.5) node[above]{}
		(12.3,0) node[rotate=270]{$\ldotp\ldotp\ldotp\ldotp$}
		(11.5,0) node[above]{}--(12.5,-0.5) node[above]{}
		(10,-1) node[rotate=270]{$\longleftrightarrow$}
		(10,-2) node[above=0.2cm]{$w + w'$} circle(0.5ex)
		(9,-1.5) node[above]{}--(10,-2) node[above]{}
		(9.2,-2) node[rotate=270]{$\cdots$}
		(9,-2.5) node[above]{}--(10,-2) node[above]{}
		(10,-2) node[above]{}--(11,-1.5) node[above]{}
		(10.8,-2) node[rotate=270]{$\ldotp\ldotp\ldotp\ldotp$}
		(10,-2) node[above]{}--(11,-2.5) node[above]{};
	\end{tikzpicture}
	\caption{Neumann moves}
	\label{fig:Neumann}
\end{figure}

Let
\[
W = \pmat{w_1 & 1 & 1 & 1 & 0 & 0 \\
	1 & w_2 & 0 & 0 & 1 & 1 \\
	1 & 0 & w_3 & 0 & 0 & 0 \\
	1 & 0 & 0 & w_4 & 0 & 0 \\
	0 & 1 & 0 & 0 & w_5 & 0 \\
	0 & 1 & 0 & 0 & 0 & w_6 }
\]
be the adjacency matrix of $ \Gamma $.
Let $ S \in \Sym_2(\Z) $ be the upper left $ 2 \times 2 $-submatrix of $ -W^{-1} $.
This is the same notation as \cite{GPPV,MM}. 
Let $ Q(m, n) \coloneqq (m, n) S {}^t\!(m, n) $ be a quadratic form and
\[
M \coloneqq w_3 w_4, \quad
N \coloneqq w_5 w_6, \quad
a \coloneqq -w_2 w_5 w_6 + w_5 + w_6, \quad
c \coloneqq -w_1 w_3 w_4 + w_3 + w_4.
\]
Let
\[
\varepsilon_\Gamma \coloneqq \sgn(Ma) \det W, \quad
S' \coloneqq \pmat{\abs{Ma} & \abs{MN} \\ \abs{MN} & \abs{Nc}}, \quad
Q'(m, n) = (m, n) S' \pmat{m \\ n}.
\]
Let $ \sigma_W $ and $ \sigma_S $ be the signatures of $ W $ and $ S $, respectively, that is, the number of positive eigenvalues minus the number of negative eigenvalues with multiplicity.

These data have the following properties.

\begin{lem} \label{lem:S}
	\begin{enumerate}
		\item \label{item:lem:S1} It holds $ \det S = MN \det W $.
		In particular, we have $ w_3, w_4, w_5, w_6 \neq 0 $ since $ \det W \neq 0 $. 
		\item \label{item:lem:S2} It holds
		\[
		S = \det W \pmat{Ma & MN \\ MN & Nc}, \quad
		Q(m, n) = \det W \left( Mam^2 + 2MNmn + Ncn^2 \right).
		\]
		\item \label{item:lem:S3}
		It holds $ \gcd(w_3, w_4) = \gcd(w_5, w_6) = 1 $.
		\item \label{item:lem:S4} It holds
		\[
		S(\Z^2)= M\Z \oplus N\Z \subset \Z^2,\quad
		(2S)^{-1}(\Z^2)/\Z^2 = \frac{1}{2M}\Z/\Z \oplus \frac{1}{2N}\Z/\Z.
		\]
		\item \label{item:lem:S:mod_k} For $ k \in \Z \smallsetminus \{ 0 \} $, $ \alpha \in \frac{1}{2M} \Z $ and  $ \beta \in \frac{1}{2N} \Z $, it holds
		\[
		Q(\alpha, \beta) \equiv Q(\alpha + k, \beta) \equiv Q(\alpha, \beta + k) \bmod k.
		\]
		\item \label{item:lem:S:sgn} The diagonal components $ Ma \det W $ and $ Nc \det W $ of $ S $ have the same sign.
		In particular, we have
		\[
		S' = \sgn(Ma) \pmat{Ma & \sgn(Na) MN \\ \sgn(Na) MN & Nc}, \quad
		Q(m, n) = \varepsilon_\Gamma Q'(m, \sgn(Na) n).
		\]
	\end{enumerate}
\end{lem}

\begin{proof}
	\cref{item:lem:S1} follows from \cite[Proposition 2.2 (ii)]{M_plumbed} or a direct calculation.
	Since
	\[
	S^{-1} = \pmat{c/M & -1 \\ -1 & a/N}
	\]
	by \cite[Proposition 2.2 (iii)]{M_plumbed} or a direct calculation, we obtain \cref{item:lem:S2}.
	To prove \cref{item:lem:S3}, we remark that \cref{item:lem:S1,item:lem:S2} and the assumption $ \det W = \pm 1 $ implies $ ac - MN = 1/\det W = \det W $.
	Since $ \gcd(w_3, w_4) $ divides both $ c $ and $ M $, it divides $ ac - MN = \pm 1 $ and thus we obtain $ \gcd(w_3, w_4) = 1 $. 
	By the same argument, we have $ \gcd(w_5, w_6) = 1 $.
	For \cref{item:lem:S4}, let
	\[
	A \coloneqq \pmat{c & -N\\ -M & a} \in \GL_2(\Z).
	\]
	Since it holds
	\[
	SA = \pmat{M & 0 \\ 0 & N},
	\]
	we have the statements.
	\cref{item:lem:S:mod_k} follows from 
	\[
	Q(\alpha + k, \beta) - Q(\alpha, \beta)
	= \det W \left( 2M\alpha ak + Ma k^2 + 2MN \beta k \right)
	\equiv 0 \bmod k
	\]
	and the same calculation for $ Q(\alpha, \beta + k) $.
	Since $ MNac = M^2 N^2 + MN \det W \ge M^2 N^2 - MN > 0 $ by \cref{item:lem:S1,item:lem:S2}, we have $ \sgn(Ma \det W) = \sgn(Nc \det W) $, which implies \cref{item:lem:S:sgn}.
\end{proof}

We prepare more notations.
Let
\[
\calS \coloneqq \calT \times \calU, \quad
\calT \coloneqq \left\{ \frac{1}{2} + \frac{e_3}{2w_3} + \frac{e_4}{2w_4} \relmiddle{|} e_3, e_4 \in \{ \pm 1 \} \right\}, \quad
\calU \coloneqq \left\{ \frac{1}{2} + \frac{e_5}{2w_5} + \frac{e_6}{2w_6} \relmiddle{|} e_5, e_6 \in \{ \pm 1 \} \right\}
\]
be sets,
$ \veps \colon \calS \to \Z $, $ \chi \colon \calT \to \Z $, and $ \psi \colon \calU \to \Z $ be maps defined as
\[
\veps(\alpha, \beta) \coloneqq \chi(\alpha) \psi(\beta), \quad
\chi( \alpha ) \coloneqq \sum_{\substack{
		e_3, e_4 \in \{ \pm 1 \} , \\
		\alpha = \frac{1}{2} + \frac{e_3}{2w_3} + \frac{e_4}{2w_4}
}} e_3 e_4, \quad
\psi (\beta) \coloneqq \sum_{\substack{
		e_5, e_6 \in \{ \pm 1 \} , \\
		\alpha = \frac{1}{2} + \frac{e_5}{2w_5} + \frac{e_6}{2w_6}
}} e_5 e_6,
\]
and $ G(q) $ and $ H(q) $ be rational functions defined as
\[
G(q) \coloneqq \frac{(q^{w_3} - q^{-w_3})(q^{w_4} - q^{-w_4})}{q^{M} - q^{-M}}, \quad
H(q) \coloneqq \frac{(q^{w_5} - q^{-w_5})(q^{w_6} - q^{-w_6})}{q^{N} - q^{-N}}.
\]

The sets $ \calT, \, \calU $, the maps $ \chi, \psi $ and the rational functions $ G(q), H(q) $ have the following properties.

\begin{rem} \label{rem:G(q)_calS_property}
	\begin{enumerate}
		\item \label{item:rem:G(q)_calS_subset}
		If $ w_3 \notin \{ \pm 1 \} $ or $ w_4 \notin \{ \pm 1 \} $, then it holds
		\[
		\calT \subset \frac{1}{2M} \Z \smallsetminus \frac{1}{2} \Z.
		\]
		Similarly, if $ w_5 \notin \{ \pm 1 \} $ or $ w_6 \notin \{ \pm 1 \} $, then it holds
		\[
		\calU \subset \frac{1}{2N} \Z \smallsetminus \frac{1}{2} \Z.
		\]
		\item \label{item:rem:G(q)_calS_rep}
		If $ w_3 \notin \{ \pm 1 \} $ and $ w_4 \notin \{ \pm 1 \} $, then it holds $ \calT \subset \left[ 0, 1 \right) $
		since
		\[
		\min \calT
		=
		\frac{1}{2} - \frac{1}{2 \abs{w_3}} - \frac{1}{2 \abs{w_4}}
		=
		\frac{(\abs{w_3} - 1) (\abs{w_4} - 1) - 1}{2 \abs{M}}
		> 0.
		\]
		Similarly, if $ w_5 \notin \{ \pm 1 \} $ and $ w_6 \notin \{ \pm 1 \} $, then it holds $ \calU \subset \left[ 0, 1 \right) $.
		\item \label{item:rem:G(q)_calS_rep_Seifert}
		If $ w_3 \in \{ \pm 1 \} $ and $ w_4 \notin \{ \pm 1 \} $, then it holds
		\[
		\calT = \left\{ \pm \frac{1}{2 w_4}, 1 \pm \frac{1}{2 w_4} \right\}, \quad
		\chi \left( \pm \frac{1}{2 w_4} \right) = \mp w_3, \quad
		\chi \left( 1 \pm \frac{1}{2 w_4} \right) = \pm w_3.
		\]
		If $ w_3 \in \{ \pm 1 \} $ and $ w_4 \in \{ \pm 1 \} $, then it holds
		\[
		\calT = \left\{ \pm \frac{1}{2}, \frac{3}{2} \right\}, \quad
		\chi \left( -\frac{1}{2} \right) = w_3 w_4, \quad
		\chi \left( \frac{1}{2} \right) = -2 w_3 w_4, \quad
		\chi \left( \frac{3}{2} \right) = w_3 w_4.
		\]
		Similar statements hold for $ w_5, w_6, \calU $. 
		\item \label{item:rem:G(q)_calS_zero}
		It holds
		\[
		\sum_{\alpha \in \calT} \chi(\alpha) = 0, \quad
		\sum_{\beta \in \calU} \psi(\beta) = 0.
		\]
		\item \label{item:rem:G(q)_calS_zero_alpha}
		It holds
		\begin{align}
			\sum_{\alpha \in \calT} \chi(\alpha) \alpha 
			=
			\sum_{ e_3, e_4 \in \{ \pm 1 \} }
			e_3 e_4 \left( \frac{1}{2} + \frac{e_3}{2w_3} + \frac{e_4}{2w_4} \right)
			=
			\sum_{ e_3, e_4 \in \{ \pm 1 \} }
			\left( \frac{e_4}{2w_3} + \frac{e_3}{2w_4} \right)
			= 0.
		\end{align}
		Similarly, it holds
		\[		
		\sum_{\beta \in \calU} \psi(\beta) \beta = 0.
		\]
		\item \label{item:rem:G(q)_calS_expansion}
		For a complex number $ q $ with $ \abs{q}^{\sgn(M)} < 1 $, it holds
		\[
		G(q) = -\sum_{\alpha \in \calT} \chi(\alpha) \sum_{m=0}^{\infty} q^{2M(m + \alpha)}.
		\]
		For a complex number $ q $ with $ \abs{q}^{\sgn(N)} < 1 $, it holds
		\[
		H(q) = -\sum_{\beta \in \calU} \psi(\beta) \sum_{n=0}^{\infty} q^{2N(n + \beta)}.
		\]
	\end{enumerate}
\end{rem}


\section{Representation of the WRT invariants} \label{sec:WRT}


In this section, we use the notations for an H-graph $ \Gamma $ prepared in \cref{sec:fund_data}.
Our main result in this section is the following representation of the WRT invariants of $ M(\Gamma) $.

\begin{prop} \label{prop:WRT_rep_Bernoulli}
	\begin{align}
		\WRT_k(M(\Gamma)) 
		= \,
		&\zeta_{4k}^{3\sigma_W - \sum_{i=1}^{6}w_{i} - \sum_{i=3}^{6} 1/w_{i}}
		\bm{e}(\sigma'_W/8)
		\\
		&\frac{1}{2k^2 (\zeta_{2k}-\zeta_{2k}^{-1})}
		\sum_{\gamma \in \calS} 
		\veps(\gamma)
		\sum_{l = {}^t\!(m, n) \in [k]^2}
		\bm{e}\left( \frac{1}{k} Q(\gamma + l) \right) mn
		\\
		= \,
		&\zeta_{4k}^{3\sigma_W - \sum_{i=1}^{6}w_{i} - \sum_{i=3}^{6} 1/w_{i}}
		\bm{e}(\sigma'_W/8)
		\\
		&\frac{\sgn(Na)}{2k^2 (\zeta_{2k}-\zeta_{2k}^{-1})}
		\sum_{\gamma \in \calS} 
		\veps(\gamma)
		\sum_{l = {}^t\!(m, n) \in [k]^2}
		\bm{e}\left( \frac{\varepsilon_\Gamma}{k} Q'(\gamma + l) \right) mn,
	\end{align}
	where $ \sigma'_W \coloneqq -\sigma_W - \sigma_S + \sgn(w_3) + \sgn(w_4) +\sgn(w_5) + \sgn(w_6) $.
\end{prop}

To begin with, we prove the following representation.

\begin{prop} \label{prop:WRT_first_calculation}
	\begin{align}
		&\WRT_k(M(\Gamma)) 
		\\
		= \, &\frac{\zeta_{4k}^{3\sigma_W - \sum_{i=1}^{6}w_{i} - \sum_{i=3}^{6} 1/w_{i}} 
			\bm{e}((-\sigma_W + \sgn(w_3) + \sgn(w_4) +\sgn(w_5) + \sgn(w_6))/8)}
		{4k (\zeta_{2k}-\zeta_{2k}^{-1}) \sqrt{\abs{MN}}}
		\\
		&\sum_{l = {}^t\!(m, n) \in (\Z \smallsetminus k\Z)^2/2kS(\Z^2)}
		\bm{e} \left( -\frac{1}{4k} {}^t\!l S^{-1} l \right)
		G(\zeta_{2Mk}^{m}) H(\zeta_{2Nk}^{n}).
	\end{align}
\end{prop}

\begin{proof}
	An idea of our proof is the same as in \cite[Proposition 6.1]{MM} and \cite[Proposition 3.2]{M_plumbed}.
	Our starting point is the following representation in \cite[Equation A.12]{GPPV}:
	\begin{align} \label{eqn:WRT}
		\WRT_k(M(\Gamma))
		=
		\frac{\bm{e}(-\sigma_W/8) \zeta_{4k}^{3\sigma_W - w_1 - \cdots w_6}}
		{2 \sqrt{2k}^6 (\zeta_{2k} - \zeta_{2k}^{-1})}
		\sum_{l \in (\Z \smallsetminus kZ)^6/2k\Z^6} \zeta_{4k}^{{}^t\!l W l}
		\prod_{1 \le i \le 6} \left( \zeta_{2k}^{l_v} - \zeta_{2k}^{-l_v} \right)^{2-\deg(i)}.
	\end{align}	
	Since it holds $ \zeta_{2k}^{l_i} - \zeta_{2k}^{-l_i} = 0 $ for $ 3 \le i \le 6, \, l_i \in k\Z $, we can write
	\begin{equation} \label{eq:WRT_Gauss_sum} 
		\begin{aligned}
			&\sum_{l \in (\Z \smallsetminus k\Z)^6/2k\Z^6} \zeta_{4k}^{{}^t\!l W l}
			\prod_{1 \le i \le 6} \left( \zeta_{2k}^{l_i} - \zeta_{2k}^{-l_i} \right)^{2-\deg(i)}
			\\
			= \,
			&\sum_{l \in \left( (\Z \smallsetminus k\Z) / 2k\Z \right)^{2} \oplus (\Z/ 2k \Z)^{4}} 
			\zeta_{4k}^{{}^t\!l W l}
			\prod_{1 \le i \le 6} \left( \zeta_{2k}^{l_i} - \zeta_{2k}^{-l_i} \right)^{2-\deg(i)}.
		\end{aligned}
	\end{equation}
	Since we can write
	\begin{align}
		\prod_{3 \le j \le 6} \left( \zeta_{2k}^{l_j} - \zeta_{2k}^{-l_j} \right)
		&= 
		\prod_{3 \le j \le 6} \sum_{e_j \in \{ \pm 1 \} } e_j \zeta_{2k}^{e_j l_j}
	\end{align}
	and for $ l \in \Z^6 $ it holds
	\[
	{}^t\!l W l
	= w_1 l_1^2 + 2 l_1 l_2 + w_2 l_2^2 + \sum_{i \in \{ 1, 2\}} \sum_{j \in \{ 2i+1, 2i+2 \}} \left( w_j l_j^2 + 2 l_i l_j \right),
	\]
	the right hand side in \cref{eq:WRT_Gauss_sum} equals to
	\begin{align}
		\sum_{l_1, l_2 \in (\Z \smallsetminus k\Z) / 2k\Z}
		\zeta_{4k}^{w_1 l_1^2 + 2 l_1 l_2 + w_2 l_2^2}
		\prod_{i \in \{ 1, 2\}}
		&\frac{1}{\zeta_{2k}^{l_i} - \zeta_{2k}^{-l_i}}
		\\
		&\prod_{j \in \{ 2i+1, 2i+2 \}} \sum_{e_j \in \{ \pm 1 \} } e_j
		\sum_{l_j \in \Z/ 2k \Z}
		\zeta_{4k}^{w_j l_j^2 + 2 (l_i + e_j) l_j}.
	\end{align}
	Since the last sum for $ l_j $ is equal to 
	\[	
	\frac{\bm{e}(\sgn(w_j)/8) \sqrt{2k}}{\sqrt{\abs{w_j}}}
	\sum_{l_j \in \Z/ w_j \Z}
	\zeta_{4k w_j}^{ -\left( 2k l_j + l_i + e_j \right)^2 }
	\]
	by \cref{prop:reciprocity}, the right hand side in \cref{eq:WRT_Gauss_sum} can be written as
	\begin{equation} \label{eq:WRT_Gauss_sum2}
		\begin{aligned}
			&\frac{4k^2 \bm{e}((\sgn(w_3) + \sgn(w_4) +\sgn(w_5) + \sgn(w_6))/8)}{\sqrt{\abs{MN}}}
			\\
			&
			\sum_{l_1, l_2 \in (\Z \smallsetminus k\Z) / 2k\Z}
			\zeta_{4k}^{w_1 l_1^2 + 2 l_1 l_2 + w_2 l_2^2}
			\prod_{i \in \{ 1, 2\}}
			\frac{1}{\zeta_{2k}^{l_i} - \zeta_{2k}^{-l_i}}
			\prod_{j \in \{ 2i+1, 2i+2 \}} \sum_{e_j \in \{ \pm 1 \} } e_j
			\sum_{l_j \in \Z/ w_j \Z}
			\zeta_{4k w_j}^{ -\left( 2k l_j + l_i + e_j \right)^2 }.
		\end{aligned}
	\end{equation}	
	Since $ \gcd(w_{2i+1}, w_{2i+2}) = 1 $ for $ i \in \{ 1, 2 \} $ by \cref{lem:S} \cref{item:lem:S3}, we have
	\begin{align}
		&\prod_{j \in \{ 2i+1, 2i+2 \}} \sum_{e_j \in \{ \pm 1 \} } e_j
		\sum_{l_j \in \Z/ w_j \Z}
		\zeta_{4k w_j}^{ -\left( 2k l_j + l_i + e_j \right)^2 } \\
		= \,
		&\sum_{l'_i \in \Z/ w_{2i+1} w_{2i+2} \Z}
		\prod_{j \in \{ 2i+1, 2i+2 \}} \sum_{e_j \in \{ \pm 1 \} } e_j
		\zeta_{4k w_j}^{ -\left( 2k l'_i + l_i + e_j \right)^2 }		
	\end{align}
	by the Chinese remainder theorem. 
	By replacing $ 2k l'_i + l_i $ as $ l_i $ and applying \cref{lem:S} \cref{item:lem:S4}, the equation in the second line in \cref{eq:WRT_Gauss_sum2} is equal to
	\begin{align}
		&\sum_{(l_1, l_2) \in (\Z \smallsetminus k\Z)^2 / 2k S(\Z)^2}
		\zeta_{4k}^{w_1 l_1^2 + 2 l_1 l_2 + w_2 l_2^2}
		\prod_{i \in \{ 1, 2\}}
		\frac{1}{\zeta_{2k}^{l_i} - \zeta_{2k}^{-l_i}}
		\prod_{j \in \{ 2i+1, 2i+2 \}} \sum_{e_j \in \{ \pm 1 \} } e_j
		\zeta_{4k w_j}^{ -\left( l_i + e_j \right)^2 }\\
		= \,
		&\sum_{(l_1, l_2) \in (\Z \smallsetminus k\Z)^2 / 2k S(\Z)^2}
		\bm{e} \left( 
		\frac{1}{4k} \left( 
		w_1 l_1^2 + 2 l_1 l_2 + w_2 l_2^2
		- \left( \frac{1}{w_3} + \frac{1}{w_4} \right) l_1
		- \left( \frac{1}{w_5} + \frac{1}{w_6} \right) l_2
		\right) \right)
		\\
		&\prod_{i \in \{ 1, 2\}}
		\frac{1}{\zeta_{2k}^{l_i} - \zeta_{2k}^{-l_i}}
		\prod_{j \in \{ 2i+1, 2i+2 \}} \sum_{e_j \in \{ \pm 1 \} } e_j
		\zeta_{4k w_j}^{ -2 e_j l_i - e_j^2 }.
	\end{align}
	Since it holds
	\[
	w_1 l_1^2 + 2 l_1 l_2 + w_2 l_2^2
	- \left( \frac{1}{w_3} + \frac{1}{w_4} \right) l_1
	- \left( \frac{1}{w_5} + \frac{1}{w_6} \right) l_2
	=
	-(l_1, l_2) S^{-1} \pmat{l_1 \\ l_2},
	\]
	we obtain the claim by the definition of $ G(q) $ and $ H(q) $. 
\end{proof}

We can prove \cref{prop:WRT_rep_Bernoulli} by \cref{prop:Gauss_sum_vanish,prop:WRT_first_calculation}. 

\begin{proof}[Proof of $ \cref{prop:WRT_rep_Bernoulli} $]
	Our proof is similar to a proof of \cite[Proposition 6.5]{MM}, which deals with the case when $ W $ is negative definite. 
	By \cref{rem:G(q)_calS_property} \cref{item:rem:G(q)_calS_expansion}, we have
	\begin{align}
		G(\zeta_{2Mk}^{m} e^{-t/2M})
		&=
		-\sum_{\alpha \in \calT} \chi(\alpha) \bm{e} \left( \frac{1}{k} m \alpha \right) e^{-t\alpha}
		\sum_{m' \in \Z_{\ge 0}} \bm{e} \left( \frac{1}{k} m m' \right) e^{-tm'}
		\\
		&=
		-\sum_{\alpha \in \calT} \chi(\alpha)
		\sum_{m' \in [k]} \bm{e} \left( \frac{1}{k} m (m' + \alpha) \right) e^{-t(m' + \alpha)}
		\sum_{m'' \in \Z_{\ge 0}} e^{-tkm''}
		\\
		&=
		-\sum_{\alpha \in \calT} \chi(\alpha)
		\sum_{m' \in [k]} \bm{e} \left( \frac{1}{k} m (m' + \alpha) \right) 
		\frac{e^{-t(m' + \alpha)}}{1 - e^{-kt}}.
	\end{align}
	Define the Bernoulli polynomials $ B_i(x) $ as
	\[
	\sum_{i=0}^{\infty} \frac{B_i(x)}{i!} t^i
	\coloneqq
	\frac{t e^{xt}}{e^t - 1}.
	\]
	Then, we can write
	\begin{align}
		G(\zeta_{2Mk}^{m} e^{-t/2M})
		&=
		\sum_{\alpha \in \calT} \chi(\alpha)
		\sum_{m' \in [k]} \bm{e} \left( \frac{1}{k} m (m' + \alpha) \right) 
		\sum_{i = -1}^\infty \frac{1}{(i+1)!} B_{i+1}\left( \frac{m' + \alpha}{k} \right) (-kt)^{i}.
	\end{align}
	Since $ m \in \Z \smallsetminus k\Z $, it holds
	\begin{equation} \label{eq:residue_vanish}
		\sum_{\alpha \in \calT} \chi(\alpha)
		\sum_{m' \in [k]} \bm{e} \left( \frac{1}{k} m (m' + \alpha) \right)
		=
		\sum_{\alpha \in \calT} \chi(\alpha)
		\bm{e} \left( \frac{1}{k} m \alpha \right)
		\sum_{m' \in \Z/k\Z} \bm{e} \left( \frac{1}{k} mm' \right)
		= 0.
	\end{equation}
	Here we remark that this also follows from the fact that the coefficient of $ t^{-1} $ in $ G(\zeta_{2Mk}^{m} e^{-t/2M}) $ vanishes since it is holomorphic at $ t = 0 $. 
	Thus, we obtain
	\[
	G(\zeta_{2Mk}^{m})
	=
	\sum_{\alpha \in \calT} \chi(\alpha)
	\sum_{m' \in [k]} \bm{e} \left( \frac{1}{k} m (m' + \alpha) \right) 
	B_{1} \left( \frac{m' + \alpha}{k} \right),
	\]
	Since $ B_1(x) = x - 1/2 $, we have
	\[
	G(\zeta_{2Mk}^{m})
	=
	\frac{1}{k}
	\sum_{\alpha \in \calT} \chi(\alpha)
	\sum_{m' \in [k]} \bm{e} \left( \frac{1}{k} m (m' + \alpha) \right) (m' + \alpha)
	\]
	by \cref{eq:residue_vanish}.
	Similarly, for $ n \in \Z \smallsetminus k\Z $ we have
	\begin{align}
		H(\zeta_{2Nk}^{n})
		=
		\frac{1}{k}
		\sum_{\beta \in \calU} \psi(\beta)
		\sum_{n' \in [k]} \bm{e} \left( \frac{1}{k} n (n' + \beta) \right) (n' + \beta).
	\end{align}
	Therefore we obtain
	\begin{align}
		&\sum_{l = {}^t\!(m, n) \in (\Z \smallsetminus k\Z)^2/2kS(\Z^2)}
		\bm{e} \left( -\frac{1}{4k} {}^t\!l S^{-1} l \right)
		G \left( \zeta_k^{m/2M} \right) H \left( \zeta_k^{n/2N} \right) \\
		= \,
		&\frac{1}{k^2}
		\sum_{l = {}^t\!(m, n) \in (\Z \smallsetminus k\Z)^2/2kS(\Z^2)}
		\bm{e} \left( -\frac{1}{4k} {}^t\!l S^{-1} l \right)
		\sum_{\gamma = {}^t\!(\alpha, \beta) \in \calS} \veps(\gamma)
		\sum_{l' = {}^t\!(m', n') \in [k]^2}
		\bm{e}\left( \frac{1}{k} {}^t\!(l' + \gamma) l \right)
		(m' + \alpha)(n' + \beta).
	\end{align}
	By replacing $ l $ as $ l + 2S(l' + \gamma) $, the right hand side in this equation is written as
	\begin{equation} \label{eq:WRT_calculation}
		\frac{1}{k^2} G(2kS)
		\sum_{\gamma = {}^t\!(\alpha, \beta) \in \calS} \veps(\gamma)
		\sum_{l' = {}^t\!(m', n') \in [k]^2}
		\bm{e}\left( \frac{1}{k} Q(l' + \gamma) \right)
		(m' + \alpha)(n' + \beta),
	\end{equation}
	where
	\[
	G(2kS) \coloneqq
	\sum_{l \in \Z^2/2kS(\Z^2)}
	\bm{e} \left( -\frac{1}{2} {}^t\!l (2kS)^{-1} l \right).
	\]
	By \cref{prop:reciprocity}, we have $ G(2kS) = 2k \bm{e} (-\sigma_S/8) \sqrt{\abs{MN}} $.
	By \cref{prop:Gauss_sum_vanish,rem:G(q)_calS_property} \cref{item:rem:G(q)_calS_zero,item:rem:G(q)_calS_zero_alpha}, \cref{eq:WRT_calculation} equals to
	\[
	\frac{2 \bm{e} (-\sigma_S/8) \sqrt{\abs{MN}}}{k}
	\sum_{\gamma \in \calS} \veps(\gamma)
	\sum_{l' = {}^t\!(m', n') \in [k]^2}
	\bm{e}\left( \frac{1}{k} Q(l' + \gamma) \right) m' n'.
	\]
	Thus, we obtain the first equality.
	
	Finally, we prove the last equality.
	If we have 
	\begin{equation}\label{eq:Gauss_sum_stability_-1}
		\begin{aligned}
			&\sum_{\gamma = {}^t\!(\alpha, \beta) \in \calS} 
			\veps(\gamma)
			\sum_{l = {}^t\!(m, n) \in [k]^2}
			\bm{e}\left( \frac{1}{k} Q(\alpha + m, \beta+n) \right) mn
			\\
			= \,
			&-\sum_{\gamma = {}^t\!(\alpha, \beta) \in \calS} 
			\veps(\gamma)
			\sum_{l = {}^t\!(m, n) \in [k]^2}
			\bm{e}\left( \frac{1}{k} Q(\alpha + m, -\beta-n) \right) mn,
		\end{aligned}
	\end{equation}
	then, since we have $ Q(m, n) = \varepsilon_\Gamma Q'(m, \sgn(Na) n) $ by \cref{lem:S} \cref{item:lem:S:sgn}, we obtain
	\begin{align}
		&\sum_{\gamma \in \calS} 
		\veps(\gamma)
		\sum_{l = {}^t\!(m, n) \in [k]^2}
		\bm{e}\left( \frac{1}{k} Q(\gamma + l) \right) mn
		\\
		= \,
		&\sgn(Ma)
		\sum_{\gamma \in \calS} 
		\veps(\gamma)
		\sum_{l = {}^t\!(m, n) \in [k]^2}
		\bm{e}\left( \frac{\varepsilon_\Gamma}{k} Q'(\gamma + l) \right) mn
	\end{align}
	and the last equality.
	We prove \cref{eq:Gauss_sum_stability_-1}.
	We have
	\begin{align}
		&\sum_{\gamma \in \calS} 
		\veps(\gamma)
		\sum_{l = {}^t\!(m, n) \in [k]^2}
		\bm{e}\left( \frac{1}{k} Q(\gamma + l) \right) mn
		\\
		= \,
		&\sum_{\gamma = {}^t\!(\alpha, \beta) \in \calS} 
		\veps(\gamma)
		\sum_{l = {}^t\!(m, n) \in [k]^2}
		\bm{e}\left( \frac{1}{k} Q(\alpha + m, \beta + k-1-n) \right) m(k-1-n)
	\end{align}
	by replacing $ n $ by $ k-1-n $.
	By \cref{lem:S} \cref{item:lem:S:mod_k}, this equals to
	\[
	\sum_{\gamma = {}^t\!(\alpha, \beta) \in \calS} 
	\veps(\gamma)
	\sum_{l = {}^t\!(m, n) \in [k]^2}
	\bm{e}\left( \frac{1}{k} Q(\alpha + m, \beta -1-n) \right) m(k-1-n).
	\]
	By \cref{prop:Gauss_sum_vanish,rem:G(q)_calS_property} \cref{item:rem:G(q)_calS_zero}, this equals to
	\[
	-\sum_{\gamma = {}^t\!(\alpha, \beta) \in \calS} 
	\veps(\gamma)
	\sum_{l = {}^t\!(m, n) \in [k]^2}
	\bm{e}\left( \frac{1}{k} Q(\alpha + m, \beta -1-n) \right) mn.
	\]
	For $ \beta \in \calU $, we have $ 1-\beta \in \calU $ and $ \psi(1-\beta) = \psi(\beta) $ by the definitions of $ \calU $ and $ \psi $.
	Thus, the above expression equals to
	\[
	-\sum_{\gamma = {}^t\!(\alpha, \beta) \in \calS} 
	\veps(\gamma)
	\sum_{l = {}^t\!(m, n) \in [k]^2}
	\bm{e}\left( \frac{1}{k} Q(\alpha + m, -\beta-n) \right) mn.
	\]
	Therefore we obtain \cref{eq:Gauss_sum_stability_-1}.
\end{proof}


\section{An asymptotic formula} \label{sec:asymptotic_formula}


In this section, we prepare an asymptotic formula which connects weighted Gauss sums and radial limits of infinite series.
We use the following notation for asymptotic expansions by Poincar\'{e}. 

\begin{dfn}[Poincar\'{e}]
	Let $ L $ be a positive number, $ f \colon \R_{>0} \to \bbC $ be maps, $ t $ be a variable of $ \R_{>0} $, and $ (a_n)_{n \ge -L} $ be a family of complex numbers.
	Then, we write
	\[
	f(t) \sim \sum_{n \ge -L} a_n t^{n} \text{ as } t \to +0
	\]
	if for any positive number $ M $ there exist positive numbers $ K_M $ and $ \varepsilon $ such that
	\[
	\abs{ f(t) - \sum_{-L \le n \le M} a_n t^{n} }
	\le K_M \abs{ t^{M+1}}
	\]
	for any $ 0 < t < \varepsilon $.
	In this case, we call the infinite series $ \sum_{n \ge -L} a_n t^{n} $ as the \textbf{asymptotic expansion} of  $ f(t) $ as $ t \to +0 $. 
\end{dfn}

Our asymptotic formula is based on the following asymptotic formula, which follows from the Euler--Maclaurin summation formula. 

\begin{lem}[{\cite[Equation (44)]{Zagier_asymptotic}, \cite[Equation (2.8)]{BKM}, \cite[Lemma 2.2]{BMM_high_depth}}]
	\label{lem:Euler-Maclaurin}
	Let $ \alpha, \beta \in \R$ be vectors and $ f \colon \R^2 \to \bbC $ be a Schwartz function.	
	Then, for a variable $ t \in \R_{>0} $, an asymptotic expansion as $ t \to +0 $
	\[
	\sum_{m, n = 0}^{\infty} f(t(m+\alpha, n+\beta))
	\sim
	\sum_{i, j = -1}^{\infty} \frac{B_{i+1}(\alpha)}{(i+1)!} \frac{B_{j+1}(\beta)}{(j+1)!}
	\frac{\partial^{i+j} f}{\partial x^i \partial y^j} (0, 0) t^{i+j}
	\]
	holds, where $ B_i(x) $ is the $ i $-th Bernoulli polynomial, and
	\begin{align}
		\frac{\partial^{i+j} }{\partial x^i \partial y^j} f \coloneqq \frac{\partial^{i}}{\partial x^i} \frac{\partial^{j}}{\partial y^j} f, \quad
		\frac{\partial^{-1} f}{\partial x^{-1}} (x', y) \coloneqq -\int_{x'}^\infty f(0, y) dx.
	\end{align}
\end{lem}

Finally, we state the main result in this section.
In the following statement, we use notations $ Q', \varepsilon, \calW $ and so on, which we prepared in \cref{sec:fund_data}.

\begin{prop} \label{prop:infin_series_asymptotic}
	Let $ k \in \Z \smallsetminus \{ 0 \} $ and $ f \colon \R^2 \to \bbC $ be a Schwartz function such that $ f(0) = 1 $.
	\begin{align}
		\lim_{t \to +0} 
		\sum_{\gamma \in \calS} 
		\veps(\gamma)
		\sum_{l \in \Z_{\ge 0}^2}
		\bm{e} \left( \frac{1}{k} Q'(\gamma + l) \right)
		f(t(\gamma + l))
		= \,
		&\frac{1}{k^2}
		\sum_{\gamma \in \calS} 
		\veps(\gamma)
		\sum_{l = {}^t\!(m, n) \in [k]^2}
		\bm{e}\left( \frac{1}{k} Q'(\gamma + l) \right) mn.
	\end{align}
\end{prop}

\begin{proof}[Proof of $ \cref{prop:infin_series_asymptotic} $]
	By \cref{lem:S} \cref{item:lem:S:mod_k} and \cref{lem:Euler-Maclaurin}, we have
	\begin{align}
		&\sum_{\gamma \in \calS} 
		\veps(\gamma)
		\sum_{l \in \Z_{\ge 0}^2}
		\bm{e} \left( \frac{1}{k} Q'(\gamma + l) \right)
		f(t(\gamma + l))
		\\
		= \,
		&\sum_{\gamma \in \calS} 
		\veps(\gamma)
		\sum_{l \in [k]^2}
		\bm{e} \left( \frac{1}{k} Q'(\gamma + l) \right)
		\sum_{l' \in \Z_{\ge 0}^2}
		f(t(\gamma + l + kl'))
		\\
		\sim \,
		&\sum_{\gamma \in \calS} 
		\veps(\gamma)
		\sum_{l \in [k]^2}
		\bm{e} \left( \frac{1}{k} Q'(\gamma + l) \right)
		\sum_{i, j = -1}^{\infty} \frac{B_{i+1}(\alpha/k)}{(i+1)!} \frac{B_{j+1}(\beta/k)}{(j+1)!}
		\frac{\partial^{i+j} f}{\partial x^i \partial y^j} (0, 0) (kt)^{i+j}.
	\end{align}
	By \cref{prop:Gauss_sum_vanish}, the terms in the last sum vanishes if $ i = -1 $ or $ j = -1 $ since $ B_0(x) = 1 $ is constant.
	Thus, the limit along $ t \to +0 $ converges and equals to
	\[
	\sum_{\gamma = \transpose{(\alpha, \beta)} \in \calS} 
	\veps(\gamma)
	\sum_{l = \transpose{(m, n)} \in [k]^2}
	\bm{e} \left( \frac{1}{k} Q'(\gamma + l) \right)
	B_{1} \left( \frac{\alpha + m}{k} \right) B_{1} \left( \frac{\beta + n}{k} \right).
	\]
	Since $ B_1(x) = x - 1/2 $, this equals to
	\[
	\frac{1}{k^2} \sum_{\gamma \in \calS} 
	\veps(\gamma)
	\sum_{l = \transpose{(m, n)} \in [k]^2}
	\bm{e} \left( \frac{1}{k} Q'(\gamma + l) \right) mn
	\]
	by \cref{prop:Gauss_sum_vanish,rem:G(q)_calS_property} \cref{item:rem:G(q)_calS_zero,item:rem:G(q)_calS_zero_alpha}.
\end{proof}


\section{Zwegers' indefinite theta functions} \label{sec:Zwegers_theta}


In \cref{sec:WRT}, we represent WRT invariants as weighted Gauss sums for a quadratic form $ Q(m, n) $.
If $ Q(m, n) $ is positive definite, then weighted Gauss sums can be represented as radial limits of false theta functions by \cite[Proposition 5.3]{MM}.
When $ Q(m, n) $ is indefinite, Zwegers' indefinite theta functions are the first candidates for such functions.
In this section, we consider some kind of indefinite theta functions which arise from our indefinite H-graphs and which are similar but slightly different from Zwegers' indefinite theta functions.
We prove that its radial limits are not equal to weighted Gauss sums but $ 0 $.

Zwegers' indefinite theta functions are defined as follows.

\begin{dfn}[{\cite[Equation 8.23]{BFOR}, Zwegers~\cite[Section 2.2]{Zwegers_thesis}}]
	\label{dfn:Zwegers_theta}
	Let $ S \in \Sym_2(\Z) $ be an indefinite symmetric matrix.
	We denote $ Q(l) \coloneqq {}^t\!l S l $ for $ l \in \R^2 $.
	Let $ \lambda, \lambda' \in \R^2 $ be two vectors such that $ Q(\lambda), Q(\lambda'), {}^t\!\lambda S \lambda' < 0 $.
	Let $ \tau $ be a complex variable with a positive imaginary part and $ q \coloneqq e^{2\pi\iu\tau} $.
	Let $ \gamma, \delta \in \R^2 $ and put $ z = \tau \gamma + \delta \in \bbC^2 $.
	Then, Zwegers' indefinite theta functions is the series
	\begin{align}
		\vartheta_{S, \lambda, \lambda'} \left( z; \tau \right)
		&\coloneqq
		\sum_{l \in \Z^2}
		\left(\sgn_0 \left( {}^t\!\lambda S (\gamma + l) \right) - \sgn_0( {}^t\!\lambda' S (\gamma + l) ) \right)
		\bm{e} \left( {}^t\!z S l \right) q^{Q(l)/2}
		\\
		&=
		\bm{e} \left( -{}^t\!\gamma S \delta \right) q^{-Q(\gamma)/2}		
		\sum_{l \in \gamma + \Z^2}
		\left(\sgn_0( {}^t\!\lambda S l ) - \sgn_0( {}^t\!\lambda' S l ) \right)
		\bm{e} \left( {}^t\!\delta S l \right) q^{Q(l)/2}.
	\end{align}
\end{dfn}

This series has the following properties.

\begin{thm}[{\cite[Theorem 8.26]{BFOR}, Zwegers~\cite[Proposition 2.4]{Zwegers_thesis}}]
	Zwegers' indefinite theta functions converge.
\end{thm}

\begin{thm}[{\cite[Theorem 8.30]{BFOR}}]
	Under the settings in \cref{dfn:Zwegers_theta}, assume $ \lambda, \lambda' \in \Z^2 $, and they have coprime coordinates. 
	Then, $ \vartheta_{S, \lambda, \lambda'} \left( z; \tau \right) $ is a mixed mock modular form $ ( $defined in \cite[Definition 13.1]{BFOR}$ ) $ of weight $ 1 $ for some congruence subgroup.
\end{thm}

For indefinite symmetric matrices $ S $ associated to indefinite H-graphs, we can calculate the radial limits of Zwegers' indefinite theta functions as follows.

\begin{rem} \label{rem:Zwegers_lim}
	As in \cref{sec:fund_data}, let $ Q' $ be the quadratic form associated to an H-graph $ \Gamma $.
	Assume that $ Q' $ is indefinite.
	For $ \gamma \in \R^2 $, define indefinite theta function as
	\begin{align} 
		\vartheta_{Q', \gamma} \left( \tau \right)
		&\coloneqq
		\sum_{l = \transpose{(m, n)} \in \Z^2} \left( \sgn_0(m) + \sgn_0(n) \right) q^{Q'(\gamma + l)}
		\\
		&=
		\left(
		2 \sum_{l \in \Z_{\ge 1}^2}
		-
		2 \sum_{l \in \Z_{\le -1}^2}
		+
		\sum_{l \in \{ 0 \} \times \Z_{\ge 1}}
		+
		\sum_{l \in \Z_{\ge 1} \times \{ 0 \}}
		-
		\sum_{l \in \{ 0 \} \times \Z_{\le -1}}
		-
		\sum_{l \in \Z_{\le -1} \times \{ 0 \}}
		\right)
		q^{Q'(\gamma + l)}.
	\end{align}
	Thus, for any positive integer $ k $, it holds 
	\[
	\lim_{\tau \to 1/k} \vartheta_{Q', \gamma} \left( \tau \right) \left( \tau \gamma; \tau \right)
	= 0
	\]
	by \cref{prop:infin_series_asymptotic}.
\end{rem}

Thus, we cannot construct $ q $-series whose radial limits are WRT invariants of indefinite H-graphs in the same manner as  Zwegers' indefinite theta functions.
So we introduce indefinite false theta functions in the next section.



\section{False indefinite theta functions whose radial limits are WRT invariants} \label{sec:infinite_series}


At last, we define false indefinite theta functions whose radial limits are WRT invariants.
We use notations an H-graph $ \Gamma $, a symmetric matrix $ S $, a quadratic form $ Q'(m, n) $, a finite set $ \calS $ and so on, which are prepared in \cref{sec:fund_data}.
We assume that $ S $ is indefinite. 

Under these notations, we define indefinite false theta functions associated with indefinite H-graphs $ \Gamma $ as follows.

\begin{dfn} \label{dfn:HB}
	For a real number $ x $ and a complex number $ q $ with $ \abs{q} < 1 $, define
	\[
	\widehat{Z}_{\Gamma} (q)
	\coloneqq 
	q^{\left( 3\sigma_W - \sum_{i=1}^{6}w_{i} - \sum_{i=3}^{6} 1/w_{i} \right)/4} \bm{e}(\sigma'_W/8) \sgn(Na)
	\sum_{\gamma \in \calS} 
	\veps(\gamma)
	\sum_{l \in \Z_{\ge 0}^2} q^{Q'(\gamma + l)},
	\]
	where
	\[
	\sigma'_W
	\coloneqq
	-\sigma_W  - \sigma_S + \sgn(w_3) + \sgn(w_4) +\sgn(w_5) + \sgn(w_6)
	\]
	as in $ \cref{prop:WRT_rep_Bernoulli} $.
\end{dfn}

Here we remark that $ \sigma_S = 0 $ since we assume that $ S $ is indefinite.

\begin{rem} \label{rem:HB_expression}
	If $ w_3, w_4, w_5, w_6 \notin \{ \pm 1 \} $, then it holds $ \calS \subset \left[ 0, 1 \right)^2 $
	by \cref{rem:G(q)_calS_property} \cref{item:rem:G(q)_calS_rep}.
	Thus, in this case, we can write
	\begin{align}
		\widehat{Z}_{\Gamma} (q)
		=
		q^{\left( 3\sigma_W - \sum_{i=1}^{6}w_{i} - \sum_{i=3}^{6} 1/w_{i} \right)/4}
		\bm{e}(\sigma'_W/8) \sgn(Na)
		\sum_{\gamma \in \calS + \Z_{\ge 0}^2} 
		\veps(\gamma) q^{Q'(\gamma)}.
	\end{align}
	Here we remark that if $ w_i \in \{ \pm 1 \} $ for some $ 3 \le i \le 6 $, then $ M(\Gamma) $ is a Seifert manifold by \cref{rem:Neumann}.
\end{rem}

Our main result is the following statement. 

\begin{thm} \label{thm:WRT_HB}
	Under the above settings, it holds
	\[
	\lim_{q \to \zeta_k^{\varepsilon_\Gamma}} \widehat{Z}_{\Gamma} \left( q \right)
	=
	\frac{1}{2 \left(\zeta_{2k} - \zeta_{2k}^{-1} \right)} \WRT_k(M(\Gamma)).
	\]
\end{thm}

\begin{proof}
	The statement follows from \cref{prop:WRT_rep_Bernoulli,prop:infin_series_asymptotic}.
\end{proof}


\section{Examples} \label{sec:examples}


In this section, we give examples of H-graphs and an indefinite false theta function whose radial limits are WRT invariants of the Poincar\'{e} homology sphere.


\subsection{Examples of H-graphs} \label{subsec:H-graph_ex}


We give examples of weights of H-graphs with $ \det W = \pm 1 $ in \cref{tab:H-graph_ex}.
In this table, we omit weights $ (-w_1, \dots, -w_6) $ such that $ (w_1, \dots, w_6) $ is shown.

\begin{table}[htb]
	\centering
	\begin{tabular}{ccc}
		\hline\noalign{\smallskip}
		The signature of $ W $ & Definiteness of $ S $ & $ (w_1, \dots, w_6) $ \\
		\hline
		\rowcolor[gray]{0.95}
		$ (0, 6) $ & Positive & $ (-1, -3, -3, -4, -3, -4), \quad (-1, -4, -2, -5, -2, -7) $ \\
		$ (1, 5) $ & Positive & $ (-1, -2, -3, 5, -2, -3), \quad (-1, -3, -2, -7, -2, 3) $ \\
		\rowcolor[gray]{0.95}
		$ (1, 5) $ & Indefinite & $ (1, 0, -2, -5, -3, -4), \quad (-1, -4, -2, -5, -2, -5) $ \\
		$ (2, 4) $ & Positive & $ (0, -1, -4, -5, -1, -4), \quad (0, 0, -1, -2, -2, -5), $ \\
		\rowcolor[gray]{0.95}
		$ (2, 4) $ & Indefinite & $ (-1, -2, -2, 5, -3, -4), \quad (-2, -1, -2, 7, -4, -5) $ \\
		$ (2, 4) $ & Negative & $ (-1, -1, 2, -3, -3, 5), \quad (-1, -1, 2, 3, -4, -5) $ \\
		\rowcolor[gray]{0.95}
		$ (3, 3) $ & Positive & $ (0, -1, 2, 3, 3, -8), \quad (0,-1, 1, -3, 3, 5) $ \\
		$ (3, 3) $ & Indefinite & $ (1, 1, -2, 3, -3, 2), \quad (1, 1, -2, -7, 4, 7) $ \\
		\hline\noalign{\smallskip}
	\end{tabular}		
	\caption{Examples of weights of H-graphs with various signatures of $ W $ and $ S $.}
	\label{tab:H-graph_ex}
\end{table}

\begin{rem}
	If $ W $ is positive definite (resp. negative definite), then $ S $ is positive definite (resp. negative definite). 
	The converse is not true, and indefiniteness of $ W $ (resp. $ S $) does not imply indefiniteness of $ S $ (resp. $ W $) by \cref{tab:H-graph_ex}.
	We can not determine the signatures of $ W $ and $ S $ by the signatures of $ w_1, \dots, w_6 $ by \cref{tab:H-graph_ex}.
	We can not prove or disprove whether the signatures of $ W $ and $ S $ determine the signatures of $ w_1, \dots, w_6 $.
\end{rem}


\subsection{The Poincar\'{e} homology sphere} \label{subsec:Poincare}


The Poincar\'{e} homology sphere is obtained by the graph shown in \cref{fig:Poincare}. 

\begin{figure}[htb]
	\centering
	\begin{tikzpicture}
		\node[shape=circle,fill=black, scale = 0.4] (1) at (0,0) { };
		\node[shape=circle,fill=black, scale = 0.4] (2) at (1.4,0) { };
		\node[shape=circle,fill=black, scale = 0.4] (3) at (-1,-1) { };
		\node[shape=circle,fill=black, scale = 0.4] (4) at (-1,1) { };
		
		\node[draw=none] (B1) at (0,0.4) {$ 1 $};
		\node[draw=none] (B2) at (1.4, 0.4) {$ 5 $};
		\node[draw=none] (B3) at (-0.6,1) {$ 2 $};
		\node[draw=none] (B4) at (-0.6,-1) {$ 3 $};
		
		\path [-](1) edge node[left] {} (2);
		\path [-](1) edge node[left] {} (3);
		\path [-](1) edge node[left] {} (4);
	\end{tikzpicture}
	\caption{The graph involving the Poincar\'{e} homology sphere} \label{fig:Poincare}
\end{figure}
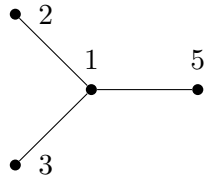

We recall previous works for discovering $ q $-series whose radial limits are WRT invariants of the Poincar\'{e} homology sphere.
The first is the false theta function discovered by Lawrence--Zagier~\cite[Theorem 1]{LZ}.
This is the same as the homological block defined by Gukov--Pei--Putrov--Vafa~\cite[Subsection 3.4]{GPPV}.
Lawrence--Zagier~\cite[Section 5]{LZ} also showed the representation of WRT invariants of the Poincar\'{e} homology sphere by Ramanujan's third order mock theta function 
\[
f(q) \coloneqq \sum_{n=0}^{\infty} \frac{q^{n^2}}{(1+q) (1+q^3) \cdots (1+q^{2n-1})}.
\]
Hikami gave the representation by Ramanujan's fifth order mock theta function 
\[
\chi_0 (q) \coloneqq \sum_{n=0}^\infty \frac{q^n}{(1-q^{n+1}) (1-q^{n+2}) \cdots (1-q^{2n})}
\]
in \cite[Proposition 1 and 2]{Hikami_mock_false} and a positive definite false theta function
\[
M_1(q) = 
\frac{1}{(q)_\infty}
\left( \sum_{m, n \ge 0} - \sum_{m, n < 0} \right)
(-1)^{m+n} q^{(3m^2 + 4mn + 3n^2 + m + n)/2}.
\]
in  \cite[Theorem 3.12]{Hikami_Hecke}.

In the rest of paper, we consider an H-graph involving the Poincar\'{e} homology sphere and prove that our indefinite false theta funtion of such an H-graph coincides with the homological block defined by Gukov--Pei--Putrov--Vafa~\cite[Subsection 3.4]{GPPV} of the Poincar\'{e} homology sphere.
We change the graph shown in \cref{fig:Poincare} to a graph shown in \cref{fig:Poincare_change} by Neumann moves (\cref{fig:Neumann}). 

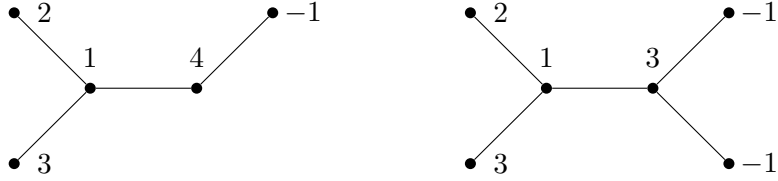
\begin{figure}[htb]
	\centering
	\begin{tikzpicture}
		\node[shape=circle,fill=black, scale = 0.4] (1) at (0,0) { };
		\node[shape=circle,fill=black, scale = 0.4] (2) at (1.4,0) { };
		\node[shape=circle,fill=black, scale = 0.4] (3) at (-1,-1) { };
		\node[shape=circle,fill=black, scale = 0.4] (4) at (-1,1) { };
		\node[shape=circle,fill=black, scale = 0.4] (5) at (2.4,1) { };
		
		\node[draw=none] (B1) at (0,0.4) {$ 1 $};
		\node[draw=none] (B2) at (1.4, 0.4) {$ 4 $};
		\node[draw=none] (B3) at (-0.6,1) {$ 2 $};
		\node[draw=none] (B4) at (-0.6,-1) {$ 3 $};
		\node[draw=none] (B5) at (2.8,1) {$ -1 $};
		
		\path [-](1) edge node[left] {} (2);
		\path [-](1) edge node[left] {} (3);
		\path [-](1) edge node[left] {} (4);
		\path [-](2) edge node[left] {} (5);
		
		\node[shape=circle,fill=black, scale = 0.4] (1) at (6,0) { };
		\node[shape=circle,fill=black, scale = 0.4] (2) at (7.4,0) { };
		\node[shape=circle,fill=black, scale = 0.4] (3) at (5,-1) { };
		\node[shape=circle,fill=black, scale = 0.4] (4) at (5,1) { };
		\node[shape=circle,fill=black, scale = 0.4] (5) at (8.4,1) { };
		\node[shape=circle,fill=black, scale = 0.4] (6) at (8.4,-1) { };
		
		\node[draw=none] (B1) at (6,0.4) {$ 1 $};
		\node[draw=none] (B2) at (7.4, 0.4) {$ 3 $};
		\node[draw=none] (B3) at (5.4,1) {$ 2 $};
		\node[draw=none] (B4) at (5.4,-1) {$ 3 $};
		\node[draw=none] (B5) at (8.8,1) {$ -1 $};
		\node[draw=none] (B6) at (8.8,-1) {$ -1 $};
		
		\path [-](1) edge node[left] {} (2);
		\path [-](1) edge node[left] {} (3);
		\path [-](1) edge node[left] {} (4);
		\path [-](2) edge node[left] {} (5);
		\path [-](2) edge node[left] {} (6);
	\end{tikzpicture}
	\caption{A new graph involving the Poincar\'{e} homology sphere} \label{fig:Poincare_change}
\end{figure}

Then, we can calculate our indefinite false theta function of this H-graph as follows.

\begin{prop} \label{prop:HB_Poincare}
	For the H-graph $ \Gamma $ shown in \cref{fig:Poincare_change}, our indefinite false theta function $ \widehat{Z}_{\Gamma} ( q ) $ has the form
	\[
	\widehat{Z}_{\Gamma} (q)
	=
	q^{-3/2}\sum_{n=-1}^{\infty} \chi_{60}(n) q^{(n^2 - 1)/120},
	\]
	where 
	\[
	\chi_{60}(n) =
	\begin{dcases}
		e_1 e_2 e_3 & \text{ if } \frac{n}{30} \equiv 1 - \frac{e_1}{2} - \frac{e_2}{3} - \frac{e_3}{5} \bmod 2\Z \text{ for some } e_1, e_2, e_3 \in \{ \pm 1 \}, \\
		0 & \text{ otherwise}.
	\end{dcases}
	\]
	Moreover, this $ q $-series coincides with the homological block of the Poincar\'{e} homology sphere.
\end{prop}

\begin{proof}
	For this $ \Gamma $, data prepared in \cref{sec:fund_data} is calculated as
	\begin{align}
		&M = 6, \quad
		N = 1, \quad
		a = -5, \quad
		c = -1, \quad
		Q'(m, n) = 30m^2 + 12mn + n^2, \\
		&\sigma_W = 0, \quad
		\sigma_S = 0, \quad
		\sigma'_W = 0, \quad
		\frac{1}{4} \left( \sum_{i=1}^{6}w_{i} + \sum_{i=3}^{6} 1/w_{i} \right)
		= \frac{3}{2} - \frac{1}{24},
		\\
		&\calT = \left\{ \frac{1}{2} \pm \frac{1}{4} \pm \frac{1}{6} \right\} 
		= \left\{ \frac{r}{12} \relmiddle{|} 0 \le r < 12, \, \gcd(r, 12) = 1 \right\}, \quad
		\calU = \left\{ \frac{1}{2} \pm \frac{1}{2} \pm \frac{1}{2} \right\} 
		= \left\{ \pm \frac{1}{2}, \frac{3}{2}  \right\},
		\\
		&\chi \left( \frac{r}{12} \right)
		= \chi_{12} (r)
		\coloneqq
		\begin{dcases}
			ee' & \text{ if } \frac{r}{6} \equiv 1 + \frac{e}{2} + \frac{e'}{3} \bmod 2\Z \text{ for some } e, e' \in \{ \pm 1 \}, \\
			0 & \text{ otherwise},
		\end{dcases}
		\\
		&\psi \left( -\frac{1}{2} \right)
		= \psi \left( \frac{3}{2} \right)
		= 1, \quad
		\psi \left( \frac{1}{2} \right)
		= -2.
	\end{align}
	Thus, we have
	\begin{align}
		\widehat{Z}_{\Gamma} (q)
		&=
		q^{-3/2 + 1/24} 
		\sum_{0 \le r < 12} \chi_{12} (r) 
		\left( \sum_{s \in \{ -1/2, 3/2 \}} - 2\sum_{s = 1/2} \right)
		\sum_{m, n = 0}^\infty q^{Q'(m+r/12, n+s)} \\
		&=
		q^{-3/2 + 1/24} 
		\sum_{r = 0}^\infty \chi_{12} (r) 
		\left( \sum_{s \in \{ -1/2, 3/2 \}} - 2\sum_{s = 1/2} \right)
		\sum_{n = 0}^\infty q^{Q'(r/12, n+s)}.
	\end{align}
	Here, for any function $ f \colon 1/2 + \Z \to \bbC $, we have
	\begin{align}
		\left( \sum_{s \in \{ -1/2, 3/2 \}} - 2\sum_{s = 1/2} \right)
		\sum_{n = 0}^\infty f \left( n+s \right)
		&=
		\sum_{n = 0}^\infty
		\left( 
		\left( f \left( n + \frac{3}{2} \right) - f \left( n + \frac{1}{2} \right) \right)
		- \left( f \left( n + \frac{1}{2} \right) - f \left( n - \frac{1}{2} \right) \right)
		\right)
		\\
		&=
		-\sum_{e \in \{ \pm 1 \}} f \left(\frac{e}{2} \right).
	\end{align}
	Thus, we have
	\[
	\widehat{Z}_{\Gamma} (q)
	=
	-q^{-3/2 + 1/24} 
	\sum_{r = 0}^\infty 
	\sum_{e \in \{ \pm 1 \}} \chi_{12} (r) e q^{Q'(r/12, e/2)}.
	\]
	Since $ Q'(m, n) = 30(m+n/5)^2 - n^2/5 $, we have
	\begin{align}
		\widehat{Z}_{\Gamma} (q)
		&=
		-q^{-3/2 + 1/24} 
		\sum_{r = 0}^\infty 
		\sum_{e \in \{ \pm 1 \}} \chi_{12} (r) e q^{30(r/12 + e/10)^2 - 1/20}
		\\
		&=
		-q^{-3/2} 
		\sum_{r = 0}^\infty 
		\sum_{e \in \{ \pm 1 \}} \chi_{12} (r) e q^{((5r + 6e)^2 - 1)/120}
		\\
		&=
		q^{-3/2} 
		\sum_{n = -1}^\infty \chi_{60} (n) q^{(n^2 - 1)/120}.
	\end{align}
	
	Next, we calculate the homological block of the Poincar\'{e} homology sphere.
	We start with the Y-graph shown in \cref{fig:Poincare}.
	Its linking matrix
	\[
	B \coloneqq 
	\pmat{
		1 & 1 & 1 & 1 \\
		1 & 2 & 0 & 0 \\
		1 & 0 & 3 & 0 \\
		1 & 0 & 0 & 5
	}
	\]
	is not negative definite, but $ (1, 1) $-component of $ B^{-1} $ is negative.
	Thus, the Y-graph is weakly negative definite (\cite[Definition 4.3]{GM}).
	For weakly negative definite plumbing graph, Gukov--Manolescu~\cite[Equation (41)]{GM} gave the definition of the homological block.
	The homological block of the Y-graph has the form
	\[
	(-1)^{\pi_B} q^{(3 \sigma_B - \tr B)/4}
	\sum_{\ell \in (1, 1, 1, 1) + 2\Z^4} F_\ell q^{-\transpose{\ell} B^{-1} \ell/4},
	\]
	where $ \pi_B $ is the number of positive eigenvalues of $ B $, $ \sigma_B $ is the signature of $ B $, $ F_\ell = F_{\ell_0} F_{\ell_1} F_{\ell_2} F_{\ell_3} $ for $ \ell = (\ell_0, \ell_1, \ell_2, \ell_3) $ and
	\[
	F_{\ell_0} \coloneqq
	\frac{1}{2} \sgn(\ell_0), \quad
	F_{\ell_i} \coloneqq
	\begin{cases}
		-\ell_i & \text{ if } \ell_i \in \{ \pm 1 \}, \\
		0 & \text{ otherwise} 
	\end{cases}
	\]
	for $ 1 \le i \le 3 $.
	Here, we have $ \pi_B = 3, \sigma_B = 2 $ and 
	\[
	-\transpose{\ell} B^{-1} \ell
	=
	30 \left( \ell_0 - \frac{\ell_1}{2} - \frac{\ell_2}{3} - \frac{\ell_3}{5} \right)^2
	- \frac{\ell_1^2}{2} - \frac{\ell_2^2}{3} - \frac{\ell_3^2}{5}.
	\]
	Thus, the homological block of the Y-graph equals to
	\[
	-q^{-5/4} 
	\sum_{\ell = (\ell_0, \ell_1, \ell_2, \ell_3) \in (1 + 2\Z) \times \{ \pm 1 \}^3}
	(-\ell_1 \ell_2 \ell_3) \frac{1}{2} \sgn(\ell_0) q^{\left( 30(\ell_0 - \ell_1/2 - \ell_2/3 - \ell_3/5)^2 - 1 - 1/30 \right)/4}.
	\]
	Since the summand is invariant under $ \ell \to -\ell $, this equals to
	\[
	q^{-3/2} 
	\sum_{e_1, e_2, e_3 \in \{ \pm 1 \}}
	\sum_{m \in \Z}
	e_1 e_2 e_3 q^{15(2m + 1 - e_1/2 - e_2/3 - e_3/5)^2/2 - 1/120}.
	\]
	This is the same $ q $-series as our $ \widehat{Z}_{\Gamma} (q) $.
\end{proof}


\bibliographystyle{alpha}
\bibliography{H-graph_indefinite}

\begin{thebibliography}{BKMN21}

\bibitem[ADH88]{Andrews-Dyson-Hickerson}
G.~E. Andrews, F.~J. Dyson, and D.~Hickerson.
\newblock Partitions and indefinite quadratic forms.
\newblock {\em Invent. Math.}, 91(3):391--407, 1988.

\bibitem[AH06]{Andersen-Hansen}
J.~E. Andersen and S.~K. Hansen.
\newblock Asymptotics of the quantum invariants for surgeries on the figure 8
  knot.
\newblock {\em Journal of Knot theory and its Ramifications}, 15(04):479--548,
  2006.

\bibitem[AH12]{Andersen-Himpel}
J.~E. Andersen and B.~Himpel.
\newblock The {W}itten--{R}eshetikhin--{T}uraev invariants of finite order
  mapping tori {II}.
\newblock {\em Quantum Topol.}, 3(3-4):377--421, 2012.

\bibitem[AM22]{Andersen-Mistegard}
J.~E. Andersen and W.~Misteg{\aa}rd.
\newblock Resurgence analysis of quantum invariants of {S}eifert fibered
  homology spheres.
\newblock {\em Journal of the London Mathematical Society}, 105(2):709--764,
  2022.

\bibitem[And13]{Andersen}
J.~E. Andersen.
\newblock The {W}itten--{R}eshetikhin--{T}uraev invariants of finite order
  mapping tori {I}.
\newblock {\em J. Reine Angew. Math.}, 681:1--38, 2013.

\bibitem[AP19]{Andersen-Petersen}
J.~E. Andersen and W.~Petersen.
\newblock Asymptotic expansions of the {W}itten--{R}eshetikhin--{T}uraev
  invariants of mapping tori {I}.
\newblock {\em Transactions of the American Mathematical Society},
  372(8):5713--5745, 2019.

\bibitem[BFOR17]{BFOR}
K.~Bringmann, A.~Folsom, K.~Ono, and L.~Rolen.
\newblock {\em Harmonic {M}aass forms and mock modular forms: theory and
  applications}, volume~64 of {\em American Mathematical Society Colloquium
  Publications}.
\newblock American Mathematical Society, Providence, RI, 2017.

\bibitem[BKM19]{BKM}
K.~Bringmann, J.~Kaszian, and A.~Milas.
\newblock Higher depth quantum modular forms, multiple {E}ichler integrals, and
  {$\mathfrak{sl}_3$} false theta functions.
\newblock {\em Res. Math. Sci.}, 6(2):Paper No. 20, 41, 2019.

\bibitem[BKMN21]{BKMN_False_modular}
K.~Bringmann, J.~Kaszian, A.~Milas, and C.~Nazaroglu.
\newblock Higher depth false modular forms.
\newblock {\em arXiv preprint arXiv:2109.00394}, 2021.

\bibitem[BMM20]{BMM_high_depth}
K.~Bringmann, K.~Mahlburg, and A.~Milas.
\newblock Higher depth quantum modular forms and plumbed 3-manifolds.
\newblock {\em Lett. Math. Phys.}, 110(10):2675--2702, 2020.

\bibitem[BN19]{Bringmann-Nazaroglu}
K.~Bringmann and C.~Nazaroglu.
\newblock A framework for modular properties of false theta functions.
\newblock {\em Res. Math. Sci.}, 6(3):Paper No. 30, 23, 2019.

\bibitem[BN22]{Bringmann-Nazaroglu_real_quad_double_sum}
K.~Bringmann and C.~Nazaroglu.
\newblock Quantum modular forms from real quadratic double sums.
\newblock {\em arXiv preprint arXiv:2205.02643}, 2022.

\bibitem[BW05]{Beasley-Witten}
C.~Beasley and E.~Witten.
\newblock Non-abelian localization for {C}hern-{S}imons theory.
\newblock {\em J. Differential Geom.}, 70(2):183--323, 2005.

\bibitem[Cha16]{Charles}
L.~Charles.
\newblock On the {W}itten asymptotic conjecture for {S}eifert manifolds.
\newblock {\em arXiv:1605.04124}, 2016.

\bibitem[Chu17]{Chun}
S.~Chun.
\newblock A resurgence analysis of the {$SU(2)$} {C}hern-{S}imons partition
  functions on a {B}rieskorn homology sphere {$\Sigma(2,5,7)$}.
\newblock {\em arXiv:1701.03528v1}, 2017.

\bibitem[Chu20]{Chung_Seifert}
H.-J. Chung.
\newblock {BPS} invariants for {S}eifert manifolds.
\newblock {\em Journal of High Energy Physics}, 2020(3):1--67, 2020.

\bibitem[Chu21a]{Chung_rational}
H.-J. Chung.
\newblock {BPS} invariants for 3-manifolds at rational level k.
\newblock {\em Journal of High Energy Physics}, 2021(2):1--23, 2021.

\bibitem[Chu21b]{Chung_resurgent}
H.-J. Chung.
\newblock Resurgent analysis for some 3-manifold invariants.
\newblock {\em Journal of High Energy Physics}, 2021(5):1--40, 2021.

\bibitem[Chu22]{Chung_SU(N)}
H.-J. Chung.
\newblock {BPS} invariants for a knot in {S}eifert manifolds.
\newblock {\em Journal of High Energy Physics}, 2022(12):1--24, 2022.

\bibitem[CM15a]{Charles-Marche_I}
L.~Charles and J.~March\'{e}.
\newblock Knot state asymptotics {I}: {AJ} conjecture and {A}belian
  representations.
\newblock {\em Publ. Math. Inst. Hautes \'{E}tudes Sci.}, 121:279--322, 2015.

\bibitem[CM15b]{Charles-Marche_II}
L.~Charles and J.~March\'{e}.
\newblock Knot state asymptotics {II}: {W}itten conjecture and irreducible
  representations.
\newblock {\em Publ. Math. Inst. Hautes \'{E}tudes Sci.}, 121:323--361, 2015.

\bibitem[DT07]{DT}
F.~Deloup and V.~Turaev.
\newblock On reciprocity.
\newblock {\em J. Pure Appl. Algebra}, 208(1):153--158, 2007.

\bibitem[FG91]{Freed-Gompf}
D.~S. Freed and R.~E. Gompf.
\newblock Computer calculation of {W}itten's {$3$}-manifold invariant.
\newblock {\em Comm. Math. Phys.}, 141(1):79--117, 1991.

\bibitem[FIMT21]{FIMT}
H.~Fuji, K.~Iwaki, H.~Murakami, and Y.~Terashima.
\newblock Witten–{R}eshetikhin–{T}uraev function for a knot in {S}eifert
  manifolds.
\newblock {\em Communications in Mathematical Physics}, 2021.

\bibitem[GM21]{GM}
S.~Gukov and C.~Manolescu.
\newblock A two-variable series for knot complements.
\newblock {\em Quantum Topology}, 12(1), 2021.
\newblock arXiv:1904.06057.

\bibitem[GMnP16]{GMP}
S.~Gukov, M.~Mari\~{n}o, and P.~Putrov.
\newblock {R}esurgence in complex {C}hern-{S}imons theory.
\newblock 2016.
\newblock arXiv:1605.07615v2.

\bibitem[GPPV20]{GPPV}
S.~Gukov, D.~Pei, P.~Putrov, and C.~Vafa.
\newblock B{PS} spectra and 3-manifold invariants.
\newblock {\em J. Knot Theory Ramifications}, 29(2):2040003, 85, 2020.

\bibitem[GZ21]{Garoufalidis-Zagier}
S.~Garoufalidis and D.~Zagier.
\newblock Knots, perturbative series and quantum modularity.
\newblock {\em arXiv preprint arXiv:2111.06645}, 2021.

\bibitem[Hik05a]{Hikami_mock_false}
K.~Hikami.
\newblock Mock (false) theta functions as quantum invariants.
\newblock {\em Regul. Chaotic Dyn.}, 10(4):509--530, 2005.

\bibitem[Hik05b]{H_Bries}
K.~Hikami.
\newblock On the quantum invariant for the {B}rieskorn homology spheres.
\newblock {\em Internat. J. Math.}, 16(6):661--685, 2005.

\bibitem[Hik05c]{H_Lattice}
K.~Hikami.
\newblock Quantum invariant, modular form, and lattice points.
\newblock {\em Int. Math. Res. Not.}, (3):121--154, 2005.

\bibitem[Hik06a]{H_Seifert}
K.~Hikami.
\newblock On the quantum invariants for the spherical {S}eifert manifolds.
\newblock {\em Comm. Math. Phys.}, 268(2):285--319, 2006.

\bibitem[Hik06b]{Hikami_AGid}
K.~Hikami.
\newblock {$q$}-series and {$L$}-functions related to half-derivatives of the
  {A}ndrews--{G}ordon identity.
\newblock {\em Ramanujan J.}, 11(2):175--197, 2006.

\bibitem[Hik06c]{H_Lattice2}
K.~Hikami.
\newblock Quantum invariants, modular forms, and lattice points. {II}.
\newblock {\em J. Math. Phys.}, 47(10):102301, 32, 2006.

\bibitem[Hik07]{Hikami_Hecke}
K.~Hikami.
\newblock Hecke type formula for unified {W}itten--{R}eshetikhin--{T}uraev
  invariants as higher-order mock theta functions.
\newblock {\em Int. Math. Res. Not. IMRN}, (7):Art. ID rnm 022, 32, 2007.

\bibitem[HK03]{Hikami-Kirillov_torus}
K.~Hikami and A.~N. Kirillov.
\newblock Torus knot and minimal model.
\newblock {\em Physics Letters B}, 575(3-4):343--348, 2003.

\bibitem[HK06]{Hikami-Kirillov_L-function}
Kazuhiro Hikami and A.~N. Kirillov.
\newblock Hypergeometric generating function of {$L$}-function, {S}later's
  identities, and quantum invariant.
\newblock {\em St. Petersburg Mathematical Journal}, 17(1):143--156, 2006.

\bibitem[HL15]{Hikami-Lovejoy_QMF}
K.~Hikami and J.~Lovejoy.
\newblock Torus knots and quantum modular forms.
\newblock {\em Research in the Mathematical Sciences}, 2(1):1--15, 2015.

\bibitem[HT04]{Hansen-Takata}
S.~K. Hansen and T.~Takata.
\newblock Reshetikhin--{T}uraev invariants of {S}eifert 3-manifolds for
  classical simple {L}ie algebras.
\newblock {\em Journal of Knot Theory and Its Ramifications}, 13(05):617--668,
  2004.

\bibitem[Jef92]{Jeffrey}
L.~C. Jeffrey.
\newblock Chern-{S}imons-{W}itten invariants of lens spaces and torus bundles,
  and the semiclassical approximation.
\newblock {\em Comm. Math. Phys.}, 147(3):563--604, 1992.

\bibitem[LO15]{Lovejoy-Osburn}
J.~Lovejoy and R.~Osburn.
\newblock Real quadratic double sums.
\newblock {\em Indag. Math. (N.S.)}, 26(4):697--712, 2015.

\bibitem[LZ99]{LZ}
R.~Lawrence and D.~Zagier.
\newblock Modular forms and quantum invariants of {$3$}-manifolds.
\newblock volume~3, pages 93--107. 1999.
\newblock Sir Michael Atiyah: a great mathematician of the twentieth century.

\bibitem[MM22]{MM}
A.~Mori and Y.~Murakami.
\newblock Witten--{R}eshetikhin--{T}uraev invariants, homological blocks, and
  quantum modular forms for unimodular plumbing {H}-graphs.
\newblock {\em SIGMA. Symmetry, Integrability and Geometry: Methods and
  Applications}, 18:034, 2022.

\bibitem[MT21]{Matsusaka-Terashima}
T.~Matsusaka and Y.~Terashima.
\newblock Modular transformations of homological blocks for {S}eifert fibered
  homology $3$-spheres.
\newblock {\em arXiv:2112.06210}, 2021.

\bibitem[Mur22]{M_plumbed}
Y.~Murakami.
\newblock Witten--{R}eshetikhin--{T}uraev invariants and homological blocks for
  plumbed homology spheres.
\newblock 2022.
\newblock arXiv:2205.01282.

\bibitem[Mur23]{M_GPPV}
Y.~Murakami.
\newblock A proof of a conjecture of {G}ukov--{P}ei--{P}utrov--{V}afa.
\newblock 2023.
\newblock arXiv:2302.13526.

\bibitem[Neu80]{Neumann_Lecture}
W.~D. Neumann.
\newblock An invariant of plumbed homology spheres.
\newblock In {\em Topology {S}ymposium, {S}iegen 1979 ({P}roc. {S}ympos.,
  {U}niv. {S}iegen, {S}iegen, 1979)}, volume 788 of {\em Lecture Notes in
  Math.}, pages 125--144. Springer, Berlin, 1980.

\bibitem[Neu81]{Neumann_work}
W.~D. Neumann.
\newblock A calculus for plumbing applied to the topology of complex surface
  singularities and degenerating complex curves.
\newblock {\em Trans. Amer. Math. Soc.}, 268(2):299--344, 1981.

\bibitem[Roz94]{Rozansky1}
L.~Rozansky.
\newblock A large {$k$} asymptotics of {W}itten's invariant of {S}eifert
  manifolds.
\newblock In {\em Proceedings of the {C}onference on {Q}uantum {T}opology
  ({M}anhattan, {KS}, 1993)}, pages 307--354. World Sci. Publ., River Edge, NJ,
  1994.

\bibitem[Roz96]{Rozansky2}
L.~Rozansky.
\newblock Residue formulas for the large {$k$} asymptotics of {W}itten's
  invariants of {S}eifert manifolds. {T}he case of {${\rm SU}(2)$}.
\newblock {\em Comm. Math. Phys.}, 178(1):27--60, 1996.

\bibitem[RT91]{Reshetikhin-Turaev}
N.~Reshetikhin and V.~G. Turaev.
\newblock Invariants of {$3$}-manifolds via link polynomials and quantum
  groups.
\newblock {\em Invent. Math.}, 103(3):547--597, 1991.

\bibitem[Wit89]{Witten}
E.~Witten.
\newblock Quantum field theory and the {J}ones polynomial.
\newblock {\em Comm. Math. Phys.}, 121(3):351--399, 1989.

\bibitem[Wu21]{Wu}
D.~H. Wu.
\newblock Resurgent analysis of {$\rm SU(2)$} {C}hern-{S}imons partition
  function on {B}rieskorn spheres {$\Sigma(2, 3, 6n + 5)$}.
\newblock {\em J. High Energy Phys.}, (2):Paper No. 008, 18, 2021.

\bibitem[Zag06]{Zagier_asymptotic}
D.~Zagier.
\newblock {\em The Mellin transform and other useful analytic techniques},
  pages 305--323.
\newblock Springer, Berlin, 2006.

\bibitem[Zag10]{Zagier_quantum}
D.~Zagier.
\newblock Quantum modular forms.
\newblock In {\em Quanta of maths}, volume~11 of {\em Clay Math. Proc.}, pages
  659--675. Amer. Math. Soc., Providence, RI, 2010.

\bibitem[Zwe02]{Zwegers_thesis}
S.~P. Zwegers.
\newblock {\em Mock theta functions}.
\newblock PhD thesis, Universiteit Utrecht, 2002.

\end{thebibliography}

\end{document}